\documentclass{article}
\usepackage{amssymb,amsmath,bm,geometry,graphics,graphicx,url,color}

\def\no{\noindent}
\def\pmatrix{\left(\begin{array}}
\def\endpmatrix{\end{array}\right)}

\newtheorem{theo}{Theorem}

\newtheorem{cor}{Corollary}
\newtheorem{rem}{Remark}
\newtheorem{defi}{Definition}

\newtheorem{assum}{Assumption}
\newtheorem{prop}{Proposation}

\newcommand{\abs}[1]{\left\vert#1\right\vert}

\newcommand{\norm}[1]{\left\Vert#1\right\Vert}

\def\sinc{\mathrm{sinc}}

\title{Long-term analysis of symplectic or symmetric extended RKN methods for nonlinear wave equations}

\author{Bin Wang\,
\footnote{Mathematisches Institut, University of T\"{u}bingen, Auf
der Morgenstelle 10, 72076 T\"{u}bingen, Germany; School of
Mathematical Sciences, Qufu Normal University, Qufu 273165, P.R.
China. The research is supported in part by the Alexander von
Humboldt Foundation and by the Natural Science Foundation of
Shandong Province (Outstanding Youth Foundation) under Grant
ZR2017JL003. E-mail:~{\tt wang@na.uni-tuebingen.de} } \and Xinyuan
Wu\thanks{School of Mathematical Sciences, Qufu Normal University,
Qufu 273165, P.R. China; Department of Mathematics, Nanjing
University, Nanjing 210093, P.R. China. The research is supported in
part by the National Natural Science Foundation of China under Grant
11671200. E-mail:~{\tt xywu@nju.edu.cn}} }

\begin{document}
\maketitle

\begin{abstract} This paper analyses the long-time behaviour of  one-stage symplectic
or symmetric extended Runge--Kutta--Nystr\"{o}m (ERKN) methods when
applied to nonlinear wave equations. It is shown that energy,
momentum, and all harmonic actions are approximately preserved over
a long time for one-stage explicit symplectic or symmetric ERKN
methods when applied to nonlinear wave equations via spectral
semi-discretisations. For the long-term analysis of symplectic or
symmetric ERKN methods, we derive a multi-frequency modulated
Fourier expansion of the ERKN method and show three
almost-invariants  of the modulation system. In the analysis of this
paper, we neither assume symmetry for symplectic methods, nor assume
symplecticity for symmetric methods. The results for symplectic and
symmetric methods are obtained as a byproduct of the above analysis.
We also give another proof by establishing a relationship between
symplectic and symmetric ERKN methods and trigonometric integrators
which have been researched for wave equations in the literature.
\medskip

\no{\bf Keywords:}  nonlinear wave equations, extended RKN methods,
multi-frequency modulated Fourier expansion, numerical conservation

\medskip
\no{\bf MSC:}35L70,  65M70,   65M15

\end{abstract}

\section{Introduction}
In this paper, we study the long-time behaviour of extended
Runge--Kutta--Nystr\"{o}m (ERKN) methods when applied to the
one-dimensional non-linear wave equation
\begin{equation}\label{wave equa}
\begin{array}[c]{ll}
u_{tt}-u_{xx}+\rho u+g(u)=0,\ \ \ t>0,\ \ -\pi\leq x\leq \pi,
\end{array}
\end{equation}
where $\rho> 0$ and the non-linearity $g$  is a smooth real function
with $g(0) = g'(0) = 0.$ We consider  periodic boundary conditions
and small initial data  in appropriate Sobolev norms. It is required
that  the initial values $u(\cdot, 0)$ and $u_t (\cdot, 0)$ are
bounded by a small parameter $\epsilon$ (see \cite{Cohen08}). This
kind of equation frequently arises in various fields of scientific
applications and many numerical methods have been developed and
researched for solving the equation (see, e.g.
\cite{Cano06,Cano13,Cano14,Gauckler15,Gauckler17-1,Gauckler18,Grimm06,Liu_Iserles_Wu(2017-2)}).

It is well known that the solution $(u(x, t), v(x, t))$ of
\eqref{wave equa} conserves several important quantities, where $v
=\partial_{t}u.$  The solution of  \eqref{wave equa} exactly
preserves the total energy
\begin{equation}\label{H}
\begin{aligned}
H(u,v)=\frac{1}{ 2 \pi}\int_{-\pi}^{\pi}\Big(
\frac{1}{2}\big(v^2+(\partial_{x}u)^2+\rho
u^2\big)(x)+U(u(x))\Big)dx,
\end{aligned}
\end{equation}
 where the potential $U(u)$ is defined as $U'(u) = g(u)$.
The momentum
\begin{equation}\label{mom}
\begin{aligned}
K(u,v)=\frac{1}{ 2
\pi}\int_{-\pi}^{\pi}\partial_{x}u(x)v(x)dx=-\sum\limits_{k=-\infty}^{\infty}\mathrm{i}ju_{-j}v_{j}
\end{aligned}
\end{equation}
is also conserved exactly by the solution of \eqref{wave equa},
where $u_j = \mathcal{F}_ju$ and $v_j = \mathcal{F}_jv$ are the
Fourier coefficients in the series $u(x)
=\sum\limits_{j=-\infty}^{\infty}u_je^{\mathrm{i}jx}$ and
 $v(x)
=\sum\limits_{j=-\infty}^{\infty}v_je^{\mathrm{i}jx}$, respectively.
It is noted that $u_{-j} = \bar{u}_j$ and $v_{-j} = \bar{v}_{j}$
since we consider only real solutions of \eqref{wave equa}. If we
rewrite the equation \eqref{wave equa} in terms of the Fourier
coefficients, the  equation $
\partial_t^2 u_{j}+\omega_j^2u_{j}+\mathcal{F}_jg(u)=0$ is obtained  with the frequencies $\omega_j=\sqrt{\rho+j^2}$ for $j\in \mathbb{Z}
$. The harmonic actions $ I_j(u,v)=\frac{\omega_j}{ 2}|u_j|^2+
\frac{1}{ 2\omega_j}|v_j|^2 $ are conserved for the linear wave
equation. For the nonlinear equation \eqref{wave equa}, it has been
shown in \cite{Bambusi03,Cohen08} that the actions remain constant
up to small deviations over a long time for almost all values of
$\rho> 0$ when the initial data is smooth and small.

For the methods applied to Hamiltonian systems of ordinary
differential equations (ODEs), the   long-time  near conservation of
the total energy and  actions  can be rigorously studied with the
help of backward error analysis (see \cite{hairer2006}). However,
for highly oscillatory Hamiltonian systems with largest frequency
$\omega$, when the product $h\omega$ is not small, no conclusion on
the long-time behaviour can be drawn by the familiar backward error
analysis. In order to overcome this restriction, Hairer and  Lubich
in \cite{Hairer00} developed a novel technique called as modulated
Fourier expansions for studying long-time conservation properties of
numerical methods for highly oscillatory Hamiltonian systems of
ODEs. It was
 extended to  multi-frequency  systems in
\cite{Cohen05} and has been used in the long-time analysis of
different  methods for various  differential equations (see, e.g.
\cite{Cohen12,Faou15,Faou14,Hairer16,hairer2006,McLachlan14,Sanz-Serna09,17-new}).
This technique    has also been used  in the analysis of wave
equations, and  we refer the reader to
 \cite{Cohen08,Cohen08-1,Gauckler16,Hairer08}.

 On the other hand, the
authors in \cite{wu2010-1}  formulated a  class  of trigonometric
integrators called as
  extended Runge--Kutta--Nystr\"{o}m (ERKN) methods for solving
muti-frequency  highly oscillatory second-order  ODEs. Recently,
further researches of trigonometric integrators on ODEs have been
done and we  refer to
\cite{Blanes2015,Buchholz2018,wang-2016,wang2018-IMA,wu2017-JCAM,wang2017-Cal,wu2013-book}.
 The ERKN methods have also been well developed
for wave equations  on  the basis of  a novel
operator-variation-of-constants formula (see, e.g.
\cite{Liu_Iserles_Wu(2017-2),JCP-Liu,AML(2017)_Liu_Wu,MeiLJ2017a,wubook2015,JCAM(2016)_Wu_Liu}).
To our knowledge, however, the long-time conservation properties of
  ERKN methods when applied to wave equations  have not been considered and  researched yet in the literature.
Under this background, in this paper, we devote to analysing the
long-time behaviour of ERKN methods when solving the nonlinear wave
equation \eqref{wave equa}.

The  main  contributions of this paper are to present  the long term
analysis not only  for symplectic ERKN methods  but also for
symmetric ERKN methods. We will show that both symplecticity and
symmetry can produce a good long-time behaviour for ERKN methods. In
the analysis, we neither assume symmetry for symplectic methods, nor
assume symplecticity for symmetric methods. This is different from
the analysis given in \cite{Cohen08-1,Gauckler16} where the symmetry
plays an important role in the construction of modulated Fourier
expansions and long-term analysis. In this paper, the long-time
behaviour of symplectic  or symmetric ERKN methods will be proved by
the technology of modulated Fourier expansion developed by Hairer
and Lubich  in \cite{Hairer00} with some novel adaptations for  ERKN
discretisation. The behaviour of symplectic and symmetric ERKN
methods can be obtained immediately as a result of the above
analysis or it can  be proved concisely by exploring the connection
between symplectic and symmetric ERKN methods and trigonometric
methods researched in \cite{Cohen08-1}.

The rest of this paper is organised as follows. In Section \ref{sec
Spectral semi-dis}, we consider a full discretisation of the
nonlinear wave equation \eqref{wave equa} by using spectral
semi-discretisation in space and ERKN methods in time. Section
\ref{sec:Main results} presents the long-time conservation
properties of
  ERKN methods and reports some  numerical results to support the
theoretical results. The two main results of symplectic or symmetric
ERKN methods are proved in Sections \ref{sec: symplectic proof}  and
\ref{sec: symmetric proof}, respectively. Section \ref{sec: ss
proof} presents another proof for the results of symplectic and
symmetric ERKN methods.
 Finally, some
concluding remarks are made in Section \ref{sec:conclusions}.

\section{Full discretisation}\label{sec Spectral semi-dis}
In order to  get a full discretisation of  \eqref{wave equa},  we
first discretise the wave equation in space by  a spectral
semi-discretisation and then in time by ERKN methods.

\subsection{Spectral semi-discretisation in space}
Following \cite{Cohen08-1,Hairer08},  the  pseudo-spectral
semi-discretisation  with equidistant collocation points $x_k = k\pi
/M$ (for $k =-M,-M+1,\ldots,M-1$) are chosen in space. Consider the
following real-valued trigonometric polynomials as an approximation
for the solution of  \eqref{wave equa}
\begin{equation}\label{trigo pol}
\begin{array}[c]{ll}
 u^{M}(x,t)=\sum\limits_{|j|\leq M}^{'}
 q_j(t)\mathrm{e}^{\mathrm{i}jx},\quad  v^{M}(x,t)=\sum\limits_{|j|\leq M}^{'}
 p_j(t)\mathrm{e}^{\mathrm{i}jx},
\end{array}
\end{equation}
where  $p_j (t) = \frac{d}{dt}q_j (t)$ and the prime indicates that
the first and last terms in the summation are taken with the factor
$1/2$. It is clear that the $2M$-periodic coefficient vector $q(t) =
(q_j (t))$ is a solution of the $2M$-dimensional system of ODEs
\begin{equation}
\frac{d^2 q}{dt^2}+\Omega^2 q=\tilde{g}(q), \label{prob}%
\end{equation}
where $\Omega$ is diagonal with entries $\omega_j$ for $|j|\leq M,$
and $\tilde{g}(q)=-\mathcal{F}_{2M}g(\mathcal{F}^{-1}_{2M}q).$ Here
$\mathcal{F}_{2M}$ denotes the discrete Fourier transform
$(\mathcal{F}_{2M}w)_j=\frac{1}{2M}\sum\limits_{k=-M}^{M-1}w_k\mathrm{e}^{-\mathrm{i}jx_k}$
 for $|j|\leq M.$ We note that the non-linearity in
\eqref{prob} can be expressed as the form
$\tilde{g}_j(q)=-\frac{\partial}{\partial q_{-j}}V(q)$  with
$V(q)=\frac{1}{2M}\sum\limits_{k=-M}^{M-1}U((\mathcal{F}^{-1}_{2M}q)_k).$
Therefore, the system \eqref{prob} is   a finite-dimensional complex
Hamiltonian system with the energy $$H_M(q,p)=\frac{1}{
2}\sum\limits_{|j|\leq M}^{'}\big(
|p_j|^2+\omega_j^2|q_j|^2\big)+V(q).$$ The actions (for $|j|\leq M$)
and the momentum of \eqref{prob} are respectively given by $$
I_j(q,p)=\frac{\omega_j}{ 2}|q_j|^2+ \frac{1}{ 2\omega_j}|p_j|^2,\ \
K(q,p)=-\sum\limits_{|j|\leq M}''\mathrm{i}jq_{-j}p_{j},$$
 where the double prime indicates that the first and
last terms in the summation are taken with the factor $1/4$. This
paper is interested in real approximation \eqref{trigo pol} and thus
we have
  $q_{-j} = \bar{q}_j$,  $p_{-j} =
\bar{p}_{j}$ and $I_{-j} = I_j$.

 In this paper, we use the same notations as those given in \cite{Cohen08-1}.
  For
$k = (k_l)_{l=0}^{\infty}$ and $\lambda =
(\lambda_l)_{l=0}^{\infty}$,  we denote
\begin{equation}\label{denot}
\begin{aligned}
|k| = (|k_l|)_{l=0}^{\infty},\quad
\norm{k}=\sum\limits_{l=0}^{\infty}|k_l|, \quad \ k\cdot
\lambda=\sum\limits_{l=0}^{\infty}k_l\lambda_l, \quad \
\lambda^{\sigma |k|}=\Pi_{l=0}^{\infty} \lambda_l^{\sigma |k_l|}
\end{aligned}
\end{equation}
for real $\sigma$.  The vector $(0, \ldots , 0, 1, 0, \ldots)$ with
the only entry at the $|j|$-th position for $j\in \mathbb{Z}$ is
denoted by  $\langle j\rangle$.  For $s\in \mathbb{R}^+$, we denote
by $H^{s}$ the Sobolev space of $2M$-periodic sequences $q=(q_j)$
endowed with the weighted norm
$\norm{q}_s=\Big(\sum\limits_{|j|\leq M}''\omega_j^{2s}
|q_j|^2\Big)^{1/2}.$

\subsection{ERKN methods in time}
As the time discretisation, we consider one-stage explicit
  extended Runge--Kutta--Nystr\"{o}m (ERKN) methods, which were
  first developed in \cite{wu2010-1}.

\begin{defi}
\label{erkn}  (See \cite{wu2010-1}) A one-stage  explicit ERKN
method for solving \eqref{prob} is defined by%
 \begin{equation}
\begin{array}
[c]{ll}%
Q^{n+c_{1}} &
=\phi_{0}(c_{1}^{2}V)q^{n}+hc_{1}\phi_{1}(c_{1}^{2}V)p^{n},\\
q^{n+1} & =\phi_{0}(V)q^{n}+h\phi_{1}(V)p^{n}+h^{2}
 \bar{b}_{1}(V)\tilde{g}(Q^{n+c_{1}}),\\
p^{n+1} & =-h\Omega^2\phi_{1}(V)q^{n}+\phi_{0}(V)p^{n}+h\textstyle
b_{1}(V)\tilde{g}(Q^{n+c_{1}}),
\end{array}
  \label{methods}%
\end{equation}
where   $h$ is a stepsize, 
$0\leq c_1\leq1$ , $b_{1}(V)$ and $\bar{b}_{1}(V)$  are
matrix-valued and uniformly bounded functions of $V\equiv
h^{2}\Omega^2$, and $
\phi_{j}(V):=\sum\limits_{k=0}^{\infty}\dfrac{(-1)^{k}V^{k}}{(2k+j)!}
$ for $j=0,1.$
\end{defi}


The following two theorems concerning the symmetry and symplecticity
of ERKN methods will be used in this paper.
\begin{theo}\label{symmetric thm}(Chap. 4  of \cite{wu2013-book})
The one-stage explicit ERKN method \eqref{methods} is symmetric if
and only if
\begin{equation}\begin{aligned}\label{sym cond}c_1=1/2,\ \
\bar{b}_{1}(V)=\phi_{1}(V)b_{1}(V)-\phi_{0}(V)\bar{b}_{1}(V),\ \
\phi_{0}(c_{1}^{2}V)\bar{b}_{1}(V)=c_{1}\phi_{1}(c_{1}^{2}V)b_{1}(V).
\end{aligned}\end{equation}
\end{theo}
\begin{theo}\label{symplectic thm} (Chap. 4 of \cite{wu2013-book})
If there exists a real number $d_1\in \mathbb{R}$ such that
\begin{equation}\begin{aligned}\label{symple cond}
&\phi_{0}(V)b_{1}(V)+V\phi_{1}(V)\bar{b}_{1}(V)=d_{1}\phi_{0}(c_{1}^{2}V),\
 \phi_{1}(V)b_{1}(V)-\phi_{0}(V)\bar{b}_{1}(V)=c_{1}d_{1}\phi_{1}(c_{1}^{2}V),
\end{aligned}\end{equation}
then the one-stage explicit ERKN method \eqref{methods} is
symplectic.
\end{theo}

It is noted that for $V=h^{2}\Omega^2$, one has $
\phi_{0}(V)=\cos(h\Omega)$ and $
\phi_{1}(V)=\textmd{sinc}(h\Omega):=\sin(h\Omega)(h\Omega)^{-1}.$
Hence, in the remainder of this paper, we use the  notations
$\bar{b}_{1}(h\Omega)$ and $ b_{1}(h\Omega)$ to denote  the
coefficients appearing in the ERKN method \eqref{methods}.

As some examples, we consider four ERKN methods whose coefficients
are displayed in Table \ref{praERKN}. According to Theorems
\ref{symmetric thm} and \ref{symplectic thm}, we know that ERKN1 is
neither symmetric nor symplectic,   ERKN2 is symplectic but not
symmetric, ERKN3 is symmetric but not symplectic, and ERKN4  is
symmetric and symplectic. By the order conditions of ERKN methods
given in \cite{wu2013-book}, it can be checked that all these four
methods are of order two.

\renewcommand\arraystretch{1.2}
\begin{table}[!htb]$$
\begin{array}{|c|c|c|c|c|c|c|c|}
\hline
\text{Methods} &c_1  &\bar{b}_1(h\Omega)   &b_1(h\Omega)  &\text{Order}& \text{Symmetric}  &  \text{Symplectic}  \\
\hline
 \text{ERKN1} & \frac{1}{2}  & \frac{1}{2} \textmd{sinc}^2(h\Omega/2)&
 \cos(h\Omega/2) & 2
& \text{No} & \text{No} \cr
\text{ERKN2} & \frac{1}{5} & \frac{4}{5} \textmd{sinc}(4h\Omega/5)  &\cos(4h\Omega/5)  & 2 & \text{No} & \text{Yes} \cr
\text{ERKN3} & \frac{1}{2} & \frac{1}{2} \textmd{sinc}^2(h\Omega/2) & \textmd{sinc}(h\Omega/2) \cos(h\Omega/2)    & 2 & \text{Yes} & \text{No} \cr
\text{ERKN4} & \frac{1}{2} & \frac{1}{2}\textmd{sinc}(h\Omega/2) & \cos(h\Omega/2)   & 2 & \text{Yes} & \text{Yes} \cr
 \hline
\end{array}
$$
\caption{Four one-stage explicit ERKN methods.} \label{praERKN}
\end{table}

\section{Main results and numerical experiment} \label{sec:Main results}

 \subsection{Main results}

In our analysis, we   make the following assumptions, where the
first four assumptions   have been considered in \cite{Cohen08-1}.
\begin{assum}\label{ass}
 $\bullet$
The initial values $q(0)$ and $p(0)$ of \eqref{prob} are assumed to
satisfy
\begin{equation}\label{initi cond}
\big(\norm{q(0)}_{s+1}^2+\norm{p(0)}_{s}^2\big)^{1/2}\leq \epsilon
\end{equation}
with a small parameter $\epsilon$.

 $\bullet$ For a given stepsize $h$, we consider the non-resonance condition
\begin{equation}
|\sin(\frac{h}{2}(\omega_j-k\cdot \omega))\cdot
\sin(\frac{h}{2}(\omega_j+k\cdot \omega))| \geq \epsilon^{1/2}h^2(
 \omega_j+|k \cdot \omega|).
 \label{inequa}%
\end{equation}
 If this condition is violated, we   define a  set of
near-resonant indices
\begin{equation}
\mathcal{R}_{\epsilon,h}=\{(j, k):|j|\leq M,\ \norm{k}\leq2N,\ \
k\neq\pm\langle j\rangle,\ \textmd{not}\ \textmd{satisfying} \
\eqref{inequa}\},
 \label{near-resonant R}%
\end{equation}
where   $N \geq 1$ is    the truncation number of the expansion
\eqref{MFE-ERKN} which will be presented in the next section.

 $\bullet$
 There exist  $\sigma
> 0$ and a constant $C_0$ such that the following non-resonance
condition is true
\begin{equation}
\sup_{(j, k)\in\mathcal{R}_{\epsilon,h}}
\frac{\omega_j^{\sigma}}{\omega^{\sigma
|k|}}\epsilon^{\norm{k}/2}\leq C_0 \epsilon^N.
 \label{non-resonance cond}%
\end{equation}

 $\bullet$
The following numerical non-resonance condition is assumed further
to be true
\begin{equation}
|\sin(h\omega_j)|\geq h \epsilon^{1/2}\ \ for \ \ |j|\leq M.
 \label{further-non-res cond}
\end{equation}

 $\bullet$ For a positive constant $c > 0$, another non-resonance condition is considered
\begin{equation}
\begin{aligned}
&|\sin(\frac{h}{2}(\omega_j-k\cdot \omega))\cdot
\sin(\frac{h}{2}(\omega_j+k\cdot \omega))| \geq c h^2 |\sinc
(h\omega_j)|\\
&\textmd{for}\ (j, k) \ \textmd{of the form}\ j=j_1+j_2\
\textmd{and}\ k=\pm\langle j_1\rangle\pm\langle j_2\rangle,
 \end{aligned}
 \label{another-non-res cond}%
\end{equation}
 which leads to improved
conservation estimates.

\end{assum}

 The main results of this paper are stated by the following two theorems.
\begin{theo}\label{main theo1} Suppose that the conditions  of Assumptions \ref{ass}  with
$s \geq \sigma + 1$ and the symplecticity conditions \eqref{symple
cond} are true (the symmetry conditions \eqref{sym cond} are not
required), the symplectic (not necessarily symmetric) ERKN method
\eqref{methods} has the near-conservation estimates
\begin{equation}
\begin{aligned}
&\sum\limits_{l=0}^{M}\omega_l^{2s+1}\frac{|I_l(q^n,p^n)-I_l(q^0,p^0)|}{\epsilon^2}
\leq
C  \epsilon,\\
&\frac{|K(q^n,p^n)-K(q^0,p^0)|}{\epsilon^2} \leq C
(\epsilon+M^{-s}+\epsilon t M^{-s+1}),\\
& \frac{|H_M(q^n,p^n)-H_M(q^0,p^0)|}{\epsilon^2} \leq C \epsilon
\end{aligned}
\label{near-conser res}%
\end{equation}
 for  actions, momentum and  energy, respectively, where    $0\leq t=nh\leq
\epsilon^{-N+1}$. The constant $C$  depends on $s, N,$ and $C_0$,
but is independent of   $\epsilon, M, h$ and the time $t=nh.$ The
bound $C \epsilon$ is weakened to $C \epsilon^{1/2}$ if condition
\eqref{another-non-res cond} is not satisfied.
\end{theo}

\begin{theo}\label{main theo2}Define the following modified
actions, modified momentum and modified energy, respectively
\begin{equation}
\begin{aligned}
&\hat{I}_j(q,p)=\sigma(h \omega_j)I_j(q,p)\qquad \quad \textmd{for}\  |j|\leq M,\\
&\hat{K}(q,p)=-\sum\limits_{|j|\leq M}''\mathrm{i}j\sigma(h \omega_j)q_{-j}p_{j},\\
& \hat{H}_M(q,p)=\frac{1}{ 2}\sum\limits_{|j|\leq M}^{'}\sigma(h
\omega_j)\big( |p_j|^2+\omega_j^2|q_j|^2\big)+V(q)
\end{aligned}
\label{MMMM}%
\end{equation}
with  $\sigma(h \omega_j)= \frac{1}{2}\textmd{sinc}\big(\frac{1}{2}h
\omega_j\big)/\bar{b}_1(h \omega_j).$ Under   the conditions  of
Assumptions \ref{ass}  with $s \geq \sigma + 1$ and the symmetry
conditions \eqref{sym cond} (the symplecticity conditions
\eqref{symple cond} are not required),  the symmetric (not
necessarily symplectic) ERKN method \eqref{methods}  has the
near-conservation estimates
\begin{equation}
\begin{aligned}
&\sum\limits_{l=0}^{M}\omega_l^{2s+1}\frac{|\hat{I}_l(q^n,p^n)-\hat{I}_l(q^0,p^0)|}{\epsilon^2}
\leq
C  \epsilon,\\
&\frac{|\hat{K}(q^n,p^n)-\hat{K}(q^0,p^0)|}{\epsilon^2} \leq C
(\epsilon+M^{-s}+\epsilon t M^{-s+1}),\\
& \frac{|\hat{H}_M(q^n,p^n)-\hat{H}_M(q^0,p^0)|}{\epsilon^2} \leq C
\epsilon.
\end{aligned}
\label{near-modi-conser res}%
\end{equation}
   If condition \eqref{another-non-res cond} is not satisfied, the
bound $C \epsilon$ is weakened to $C \epsilon^{1/2}$.
\end{theo}
\begin{rem} It is noted that similar conservations have been shown for   symmetric and symplectic
trigonometric integrators  in \cite{Cohen08-1}. In the analysis of
that paper, the symmetry and symplecticity are needed simultaneously
and they play an important role  in the construction of modulated
Fourier expansions and deriving long time conservations. However, in
this paper we get the long-term analysis for two kinds of ERKN
methods: symplectic methods or symmetric methods. In the analysis,
we neither need symmetry for symplectic methods, nor need
symplecticity for symmetric methods. We show that for symplecticity
and symmetry, any one of them can produce a good long-time behaviour
for ERKN methods and it is not necessary to require both.
\end{rem}

\begin{rem} It is also noted that for  the  symplectic and symmetric ERKN method,
one has $\sigma\equiv1$. Thus   modified actions, modified momentum
and modified energy given in \eqref{MMMM} become the normal actions,
momentum and
  energy of the considered system, respectively. Therefore, we have the following result.
\end{rem}
\begin{cor}\label{main theo3}Under the  conditions  of Assumptions \ref{ass}  with
$s \geq \sigma + 1$, the  symplectic and symmetric ERKN method
\eqref{methods} has the near-conservation estimates
\eqref{near-conser res}.
\end{cor}

 \subsection{Numerical experiment} As a numerical experiment, we consider the non-linear wave equation
\eqref{wave equa} with the following data: $\rho=0.5,\ g(u)=-u^2$,
and the initial conditions
$u(x,0)=0.1\big(\frac{x}{\pi}-1\big)^3\big(\frac{x}{\pi}+1\big)^2,\
\partial_t u(x,0)=0.01\frac{x}{\pi}\big(\frac{x}{\pi}-1\big)\big(\frac{x}{\pi}+1\big)^2$
for $-\pi\leq x\leq \pi$. This problem has been considered in
\cite{Cohen08-1}.  The spatial discretisation is \eqref{prob} with
the  dimension $2M = 2^7$.
 We apply the four ERKN methods displayed in Table \ref{praERKN}
with  the  stepsize $h = 0.5$ to this problem on  $[0,100000]$. The
errors of energy, momentum and action  $I=\sum\limits_{j}I_j$ for
ERKN1,2,4 against $t$ and the results of modified energy, momentum
and action for ERKN3
 are shown  in Figure \ref{fig0}. 

It follows from the results that  the energy, momentum and action
 are not well conserved by   ERKN1 but are very well conserved by symplectic ERKN2 or symplectic and symmetric
 ERKN4. The symmetric but not symplectic ERKN3 has a good conservations of modified  energy, momentum and action
 over long times. These observations  support the results given in Theorems \ref{main
theo1}-\ref{main theo2} and Corollary \ref{main theo3}.

\begin{figure}[ptb]
\centering\tabcolsep=2mm
\begin{tabular}
[l]{lll}%
\includegraphics[width=6cm,height=6cm]{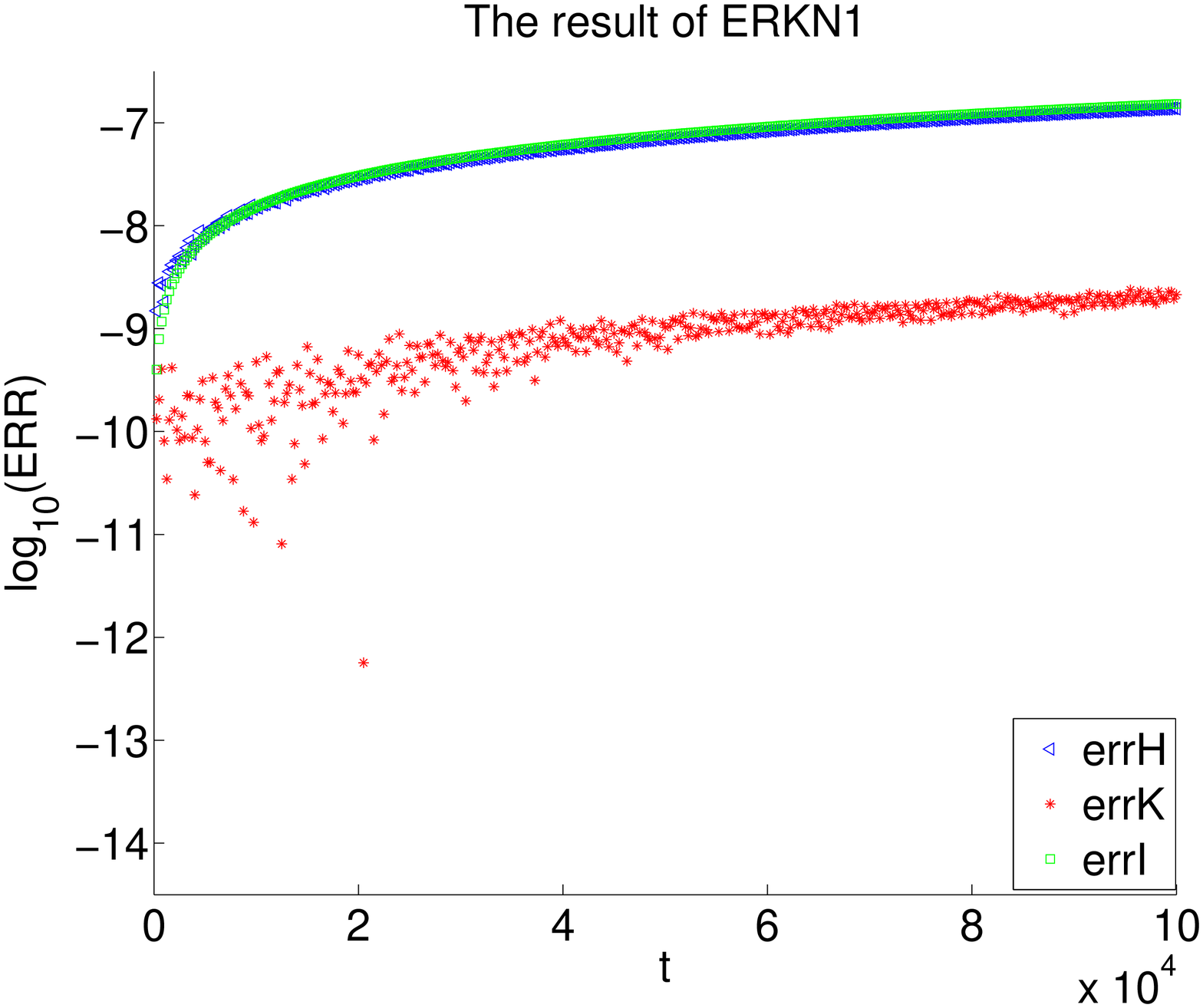}
\includegraphics[width=6cm,height=6cm]{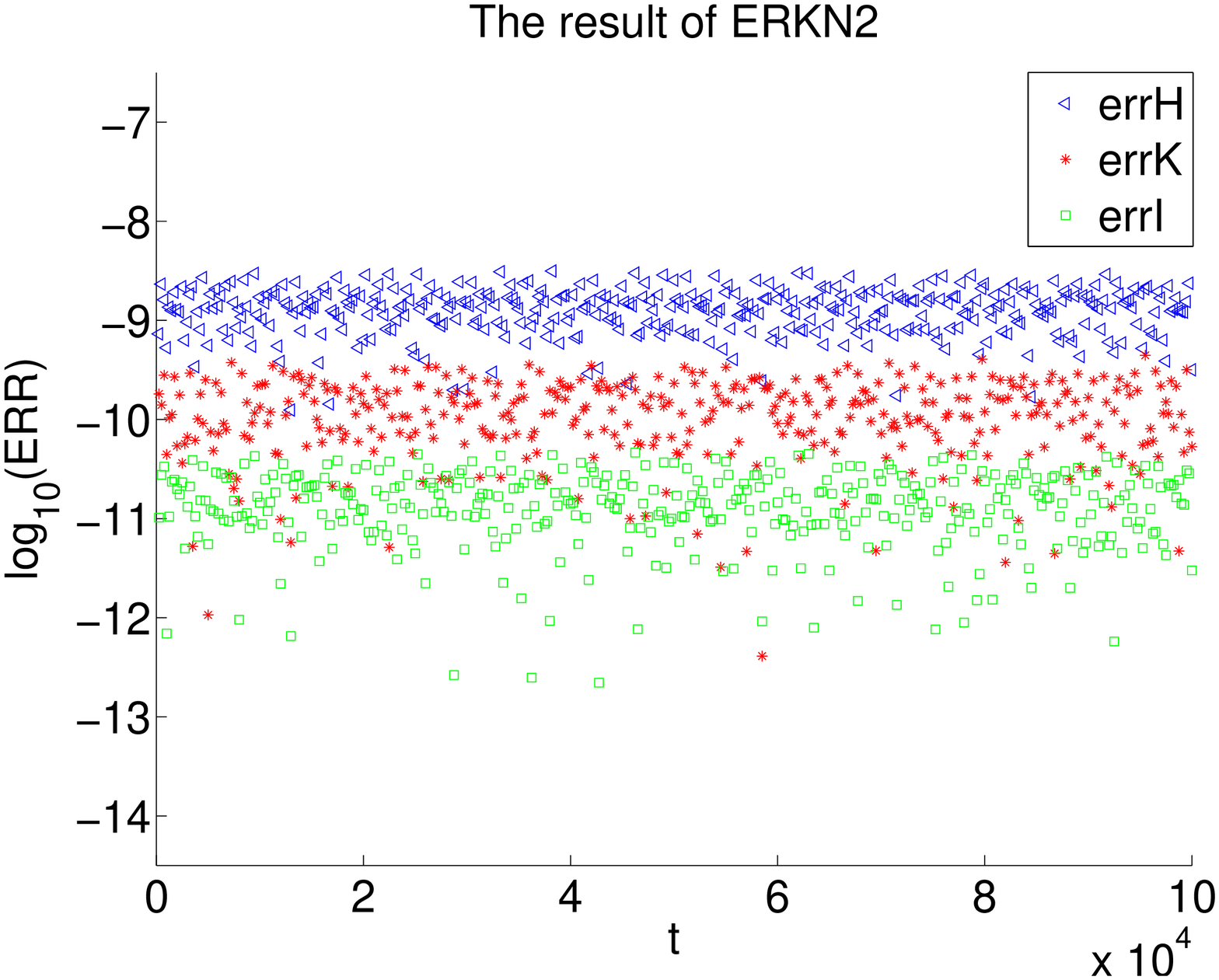}\\
\includegraphics[width=6cm,height=6cm]{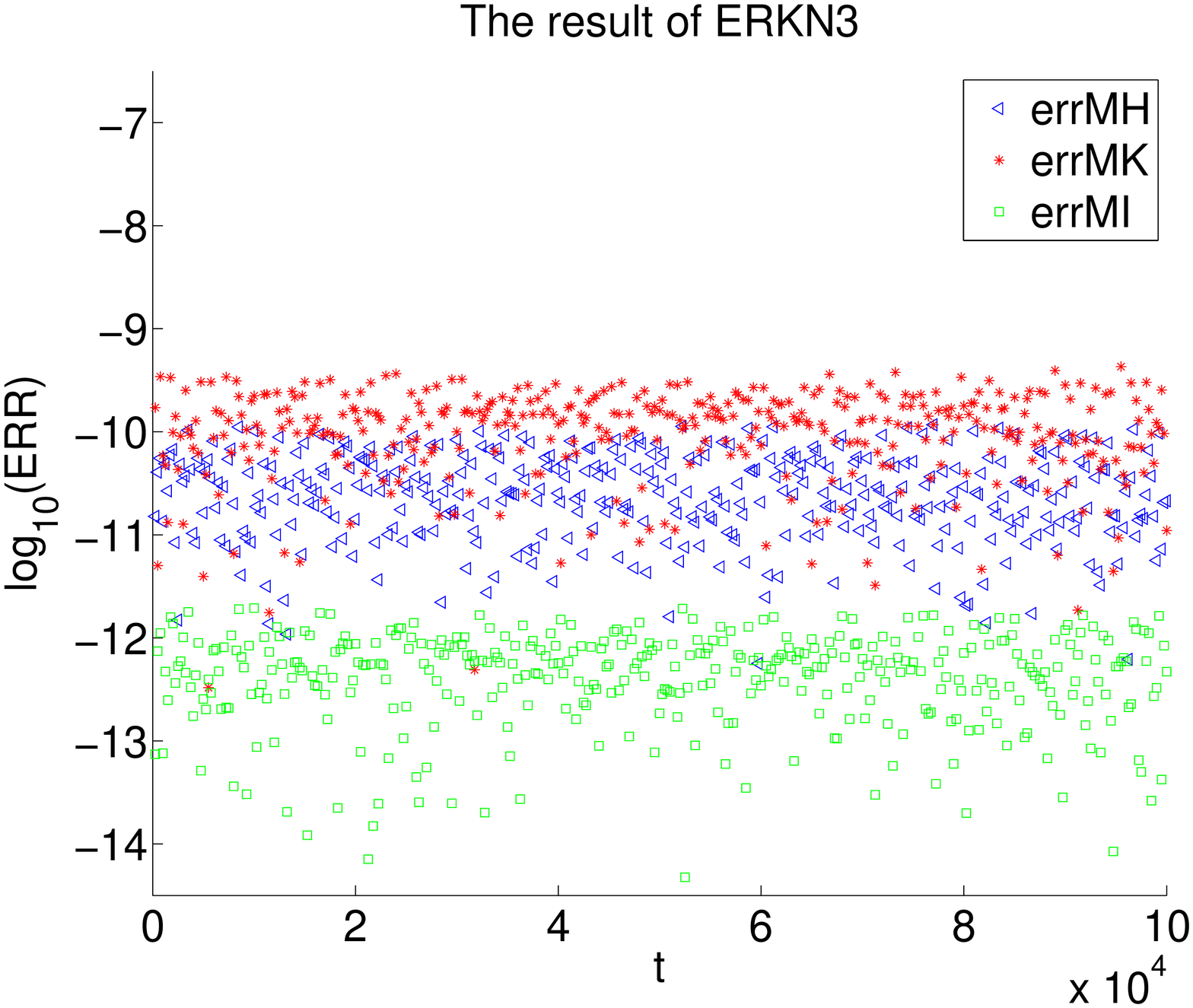}
\includegraphics[width=6cm,height=6cm]{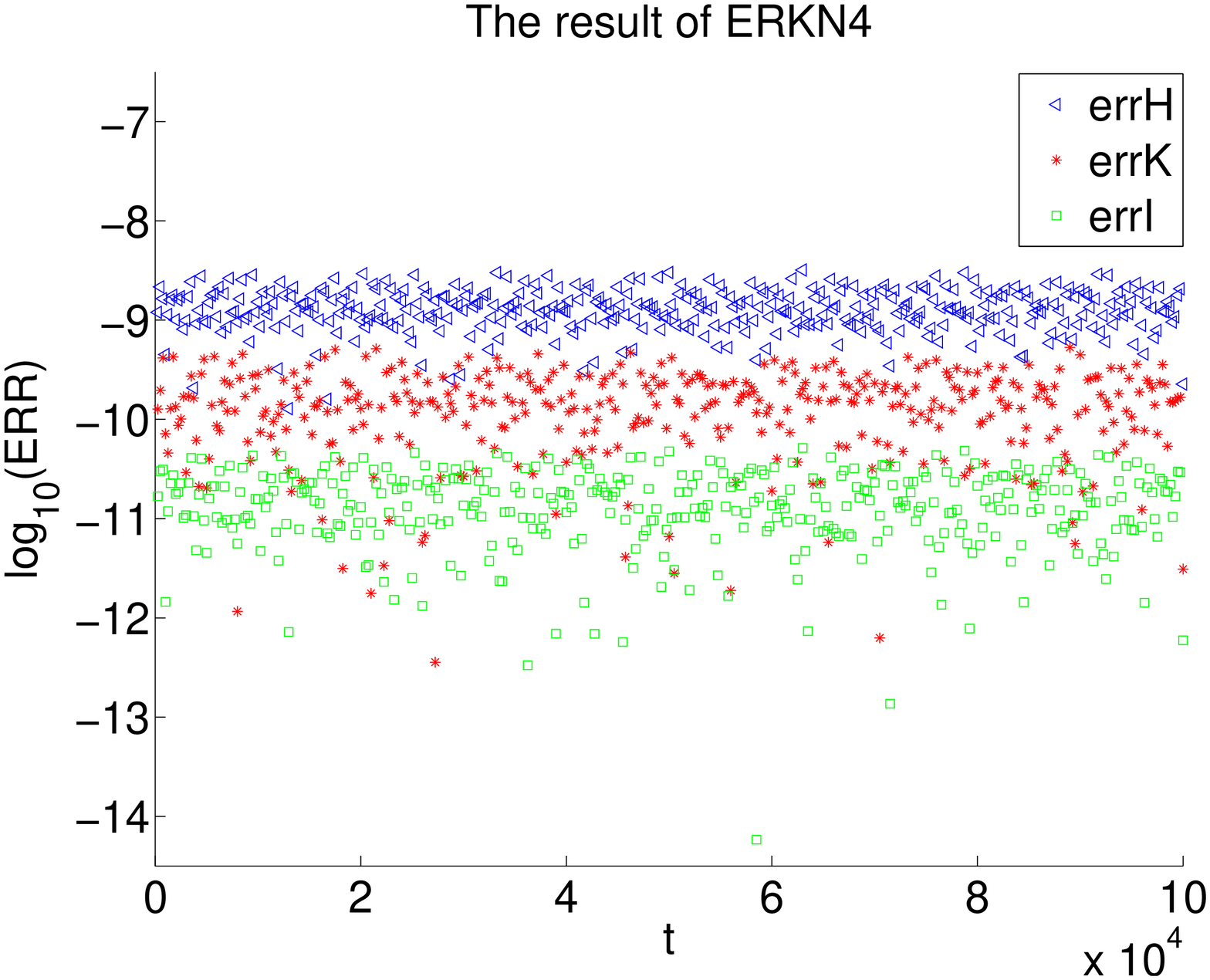}
\end{tabular}
\caption{The logarithm of the  errors against $t$.}%
\label{fig0}%
\end{figure}


\section{Proof of Theorem \ref{main
theo1}} \label{sec: symplectic proof} Theorem \ref{main theo1}  will
be proved in this section by using the technique of multi-frequency
modulated Fourier expansions. In the proof, the following   points
will be shown one by one in the rest of this section.
\begin{itemize}
\item In Sect. \ref{subsec:MEF}   the result of modulated Fourier
expansion is presented, which will be proved by Sects.
\ref{subsec:mod equ}-\ref{subsec:rem}.

\item  Sect. \ref{subsec:mod equ}  derives the  modulation equations of the modulation
functions.

\item In Sect. \ref{subsec interation}  an iterative construction of the
functions is considered by using  reverse Picard iteration.

\item A more convenient rescaling is chosen and  the estimation of  non-linear terms  is studied in Sect. \ref{subsec rescal}.

\item  Then the iteration is  reformulated in Sect. \ref{subsec:ref
rev}.

\item In Sect. \ref{subsec: bounds}  we control the size of the coefficient functions.

\item  The bound of the defect is  estimated in Sect. \ref{subsec:def}.

\item We study the
remainder of the numerical solution and its modulated Fourier
expansion in Sect.  \ref{subsec:rem}.

\item In Sect. \ref{subsec:alm inv} we show three almost-invariants of the
modulation system.

\item  We   establish the relationship between these  almost-invariants and the quantities of  the
system in  Sect.  \ref{subsec:rela}.

\item Finally, the previous results   are extended  to a long time interval in Sect.
\ref{subsec:fro to}.
\end{itemize}

Before presenting these analysis, we need to  define three operators
by
\begin{equation}\label{new L1L2}
 \begin{array}[c]{ll}
L_1^{k}:=& \big(\mathrm{e}^{\mathrm{i}(k \cdot\omega)
h}\mathrm{e}^{\epsilon
hD}-\cos(h\Omega)-h^2\Omega^2\textmd{sinc}(h\Omega)\bar{b}_1(h\Omega)b^{-1}_1(h\Omega)\big)\\
&\big(\bar{b}_1(h\Omega)b^{-1}_1(h\Omega)(\mathrm{e}^{\mathrm{i}(k
\cdot\omega) h}\mathrm{e}^{\epsilon
hD}-\cos(h\Omega))+\textmd{sinc}(h\Omega)\big)^{-1},\\
L_2^{k}= &
\cos(c_1h\Omega)+c_1\textmd{sinc}(c_1h\Omega)L_1^{k},\\
 L^{k}=
&\big(\mathrm{e}^{\mathrm{i}(k \cdot\omega) h}\mathrm{e}^{\epsilon
hD}-\cos(h\Omega)-\textmd{sinc}(h\Omega)L_1^{k}\big)
(\bar{b}_1(h\Omega)L_2^{k})^{-1} ,\end{array}\end{equation}
 where $D$ is the
differential operator (see \cite{hairer2006}). For the operator $
L^{k}$, the following result will be used in this section.
\begin{prop}\label{lhd pro}
  The  Taylor expansions of the operator $L^{k}$ are expressed as
\begin{equation}\label{L expansion}
\begin{aligned}&L^{\pm\langle j\rangle}(hD)\alpha_j^{\pm\langle j\rangle}(\epsilon t)= \pm2 \textmd{i}
\epsilon h^2   \omega_j \dot{\alpha}_j^{\pm\langle
j\rangle}(\epsilon t)+  \epsilon^2 h^3
\omega_j\cot(h\omega_j)\ddot{\alpha}_j^{\pm\langle
j\rangle}(\epsilon t)+\cdots,\\
 &L^{k}(hD)\alpha_j^{k}(\epsilon t)= 4h\omega_j\csc(h\omega_j) s_{\langle j\rangle+k}s_{\langle
j\rangle-k}\alpha_j^{k}(\epsilon t)+2\textmd{i}  \epsilon
h^2\omega_j\csc(h\omega_j) s_{2k} \dot{\alpha}_j^{k}(\epsilon t)+\cdots,\\
\end{aligned}
\end{equation}
for $\abs{j}>0$ and $k\neq \pm\langle j\rangle$, where $s_{k} =\sin(
\frac{h}{2}(k\cdot\omega))$  and the dots   denote derivatives with
respect to   $\tau=\epsilon t$.
\end{prop}

\subsection{The result of modulated Fourier expansion} \label{subsec:MEF}

Following the techniques and tools  developed in
\cite{Cohen05,Cohen08,Cohen08-1,Hairer08,17-new}, the long-time
analysis of ERKN methods when applied to  wave equation is based on
a short-time  multi-frequency modulated Fourier expansion, which
will be constructed in this part and proved in Sects.
\ref{subsec:mod equ}-\ref{subsec:rem}. In what follows,  we  set $
[[k]]=\left\{\begin{aligned} & (\norm{k}+1)/2,\quad
k\neq0,\\
&3/2,\qquad\quad\quad \ \  k=0.
\end{aligned}\right.
$

\begin{theo}\label{MFE thm}
Under the conditions of Theorem \ref{main theo1}, the numerical
solution $(q^n, p^n)$ given by
  the ERKN method
\eqref{methods} admits the following multi-frequency modulated
Fourier expansion (with $N$ from \eqref{near-resonant R})
\begin{equation}
\begin{aligned} &\tilde{q}(t)= \sum\limits_{\norm{k}\leq 2N} \mathrm{e}^{\mathrm{i}(k \cdot \omega) t}\zeta^k(\epsilon t),\
\ \tilde{p}(t)=  \sum\limits_{\norm{k}\leq2N} \mathrm{e}^{\mathrm{i}(k \cdot \omega) t}\eta^k(\epsilon t)\\
\end{aligned}
\label{MFE-ERKN}%
\end{equation}
such that
\begin{equation}
\norm{q^n-\tilde{q}(t)}_{s+1}+\norm{p^n-\tilde{p}(t)}_{s}\leq
C\epsilon^N\quad for \quad 0\leq t=nh\leq \epsilon^{-1},
\label{error MFE}%
\end{equation}
where we use the notation introduced in \eqref{denot} with $k_l = 0
$ for $l> M$. The expansion \eqref{MFE-ERKN} is bounded by
\begin{equation}
\norm{\tilde{q}(t)}_{s+1}+\norm{\tilde{p}(t)}_{s}\leq C\epsilon\quad
for \quad 0\leq t\leq \epsilon^{-1}.
\label{bound TMFE}%
\end{equation}
It is further obtained that for $|j|\leq M$,
\begin{equation}\begin{aligned}
&\tilde{q}_j(t)=\zeta_j^{\langle j\rangle}(\epsilon t)
\mathrm{e}^{\mathrm{i} \omega_j t}+\zeta_j^{-\langle
j\rangle}(\epsilon t) \mathrm{e}^{-\mathrm{i} \omega_j t}+r_j \quad
\textmd{with} \quad \norm{r}_{s+1}\leq C\epsilon^2.
\end{aligned}
\label{jth TMFE}
\end{equation}
(If the condition  \eqref{another-non-res cond} fails to be
satisfied, then the bound is $\norm{r}_{s+1}\leq C\epsilon^{3/2}$.)
The modulation functions $\zeta^k$  are bounded by
\begin{equation}\begin{aligned}
&\sum\limits_{\norm{k}\leq2N}\Big(\frac{\omega^{|k|}}{\epsilon^{[[k]]}}\norm{\zeta^k(\epsilon
t)}_s\Big)^2\leq C.
\end{aligned}
\label{bound modu func}%
\end{equation}
The same bounds      hold for any fixed number of derivatives of
 $\zeta^k$  with respect to the slow time  $\tau = \epsilon
t$. The relationship between $\zeta^k$ and $\eta^k$ is given by
\eqref{MFE-zetaeta} in next section. Moreover, we have
$\zeta_{-j}^{-k}=\bar{\zeta_{j}^{k}}$. The constant $C$ is
independent of $\epsilon, M, h$ and   $t \leq \epsilon^{-1}.$
\end{theo}

\subsection{Modulation equations of the  modulation functions}\label{subsec:mod equ}
According to the second and third formulae of \eqref{methods}, it is
arrived that
\begin{equation*}
\begin{aligned} &q^{n+1}-\big(\cos(h\Omega)+h^2\Omega^2\sinc(h\Omega)\bar{b}_1(h\Omega)b^{-1}_1(h\Omega)\big)q^{n}\\
=&h\bar{b}_1(h\Omega)b^{-1}_1(h\Omega)p^{n+1}+h\big(\sinc(h\Omega)-\bar{b}_1(h\Omega)b^{-1}_1(h\Omega)\cos(h\Omega)\big)p^{n}.
\end{aligned}
\end{equation*}
Inserting \eqref{MFE-ERKN} into this result and  comparing the
coefficients of $\mathrm{e}^{\mathrm{i}(k \cdot \omega) t}$ implies
\begin{equation}\label{MFE-zetaeta}%
\begin{aligned}& L_1^{k} \zeta^k=h\eta^k,
\end{aligned} %
\end{equation}
where the definition of $L_1^{k}$ \eqref{new L1L2} is used here.

According to the first   term  of     \eqref{methods}, we consider
the function
\begin{equation}
\begin{aligned} &\tilde{q}_{h}(t)=\sum\limits_{\norm{k}\leq 2N}
\mathrm{e}^{\mathrm{i}(k \cdot \omega) t}\xi^k(t\epsilon)
\end{aligned}
\label{MFE-2}%
\end{equation}
as an ansatz of $Q^{n+\frac{1}{2}}$ in  \eqref{methods}. Inserting
\eqref{MFE-ERKN} and \eqref{MFE-2} into the first  term
 of \eqref{methods}  and comparing the coefficients of $\mathrm{e}^{\mathrm{i}(k \cdot
\omega) t}$, we have
\begin{equation*}
\xi^k(t\epsilon)=L_2^{k}\zeta^k(t\epsilon).
\end{equation*}

Inserting \eqref{MFE-ERKN} and \eqref{MFE-2} into the second formula
of \eqref{methods}, expanding the right-hand side into a Taylor
series around zero and comparing the coefficients of
$\mathrm{e}^{\mathrm{i}(k \cdot \omega) t}$ leads to
\begin{equation}\label{ljkqp}
\begin{aligned}L^{\pm\langle j\rangle} \zeta_j^{\pm\langle j\rangle}(t\epsilon)=&-h^2\sum\limits_{m\geq 2}\frac{g^{(m)}(0)}{m!}
\sum\limits_{k^1+\cdots+k^m=\pm\langle
j\rangle}\sum\limits_{j_1+\cdots+j_m\equiv j\ \textmd{mod}\ 2M}'
\big(\zeta_{j_1}^{k^1}\cdot\ldots\cdot
\zeta_{j_m}^{k^m}\big)(t\epsilon),\\
L^k \zeta_j^k(t\epsilon)=&-h^2\sum\limits_{m\geq
2}\frac{g^{(m)}(0)}{m!}
\sum\limits_{k^1+\cdots+k^m=k}\sum\limits_{j_1+\cdots+j_m\equiv j\
\textmd{mod}\ 2M}'  \big(\zeta_{j_1}^{k^1}\cdot\ldots\cdot
\zeta_{j_m}^{k^m}\big)(t\epsilon),\\
\end{aligned}
\end{equation}
where the right-hand side is obtained  in a similar way to that in
\cite{Cohen08-1,Hairer08}. The prime on the sum indicates that a
factor $1/2$ is included in the appearance of $\zeta_{k_i}^{j_i}$
with $j_i = \pm M$.

Similarly to the analysis in Sect. 6.2 of \cite{Cohen08-1}, for the
case that $k = \pm \langle j\rangle$, the dominating term in
\eqref{L expansion} is $ \pm2 \textmd{i} \epsilon h^2   \omega_j
\dot{\zeta}_j^{\pm \langle j\rangle}$ due to the fact that the first
term vanishes and the condition  \eqref{further-non-res cond}. When
$k \neq \pm \langle j\rangle$,
 the first
term  in \eqref{L expansion}  is dominant under the condition
\eqref{inequa}.  If the inequality \eqref{inequa} fails to be true,
the non-resonance condition  \eqref{non-resonance cond}   ensures
that the defect in simply setting $ \zeta_j^k\equiv 0$ is of size
$\mathcal{O}(\epsilon^{N+1})$ in an appropriate Sobolev-type norm.
By the above analysis and \eqref{ljkqp}, formal modulation equations
are obtained for the modulated functions for $ \zeta_j^k$. Then the
 modulated functions for $ \eta_j^k$ are derived by
 \eqref{MFE-zetaeta}.

On the other hand, we need to derive the initial values
 for $\dot{\zeta}_j^{\pm \langle j\rangle}$ appearing in the first   modulation equation of \eqref{ljkqp}. From $\tilde{q} (0) =
q^0$, it follows that
\begin{equation}\label{initial pl}
\begin{aligned}&\zeta_j^{\langle
j\rangle}(0)+\zeta_j^{-\langle j\rangle}(0)=q_j^0-\sum\limits_{k\neq
\pm  \langle j\rangle} \zeta_j^{k}(0).
\end{aligned}
\end{equation}
Moreover, according to the formula \eqref{MFE-zetaeta}, we have
 \begin{equation}\label{impor-zetaeta}%
\begin{aligned}&\eta_j^{\pm\langle j\rangle}=\pm\textmd{i} \omega_j\zeta_j^{\pm\langle
j\rangle}+e^{\textmd{i}(2c_1-1)h\omega_j} \epsilon
h^2\omega_j\csc(h\omega_j)\dot{\zeta}_j^{\pm\langle
j\rangle}\pm\cdots.
\end{aligned} %
\end{equation}
Then the condition $\tilde{p} (0) = p^0$ yields
\begin{equation}\label{initial mi}
\begin{aligned}& \eta_j^{\langle
j\rangle}(0)+\eta_j^{-\langle j\rangle}(0)=p_j^0-\sum\limits_{k\neq
\pm  \langle j\rangle} \eta_j^{k}(0)\\
&= \textmd{i} \omega_j(\zeta_j^{\langle
j\rangle}(0)-\zeta_j^{-\langle
j\rangle}(0))+2e^{\textmd{i}(2c_1-1)h\omega_j} \epsilon
h^2\omega_j\csc(h\omega_j)\dot{\zeta}_j^{\pm\langle
j\rangle}(0)+\cdots.
\end{aligned}
\end{equation}
The formulae \eqref{initial pl} and \eqref{initial mi} determine the
initial values for $\zeta_j^{\pm \langle j\rangle}(0)$.

\subsection{Reverse Picard iteration}\label{subsec interation}
In this subsection, we consider   an iterative construction of the
functions $\zeta_j^k$  such that after $4N$ iteration steps, the
defects in \eqref{ljkqp}, \eqref{initial pl} and \eqref{initial mi}
are of magnitude $\mathcal{O}(\epsilon^{N+1})$ in the $H^s$ norm. In
the iteration procedure, only the dominant terms are kept on the
left-hand side, which is called as reverse Picard iteration by
following \cite{Cohen08-1,Hairer08}.

Denote  the $n$th iterate  by $[\cdot]^{(n)}$.    For $k = \pm
\langle j\rangle$, the iteration procedure is considered as
\begin{equation}\label{Pic ite j}
\begin{aligned}\pm2 \textmd{i}
\epsilon h^2   \omega_j \big[\dot{\zeta}_j^{\pm  \langle
j\rangle}\big]^{(n+1)}=&\Big[-h^2\sum\limits_{m\geq
2}\frac{g^{(m)}(0)}{m!}
\sum\limits_{k^1+\cdots+k^m=k}\sum\limits_{j_1+\cdots+j_m\equiv j\
\textmd{mod}\ 2M}'\big(\zeta_{j_1}^{k^1}\cdot\ldots\cdot
\zeta_{j_m}^{k^m}\big)  \\& -\big(\epsilon^2 h^3
\omega_j\cot(h\omega_j) \ddot{\zeta}_j^{\pm \langle
j\rangle}+\cdots\big)\Big]^{(n)}.
\end{aligned}
\end{equation}
For $k \neq\pm  \langle j\rangle$
 and
$j$ satisfying the non-resonant \eqref{inequa}, we consider the
following iteration procedure
\begin{equation}\label{Pic ite notj}
\begin{aligned}4h\omega_j\csc(h\omega_j) s_{\langle j\rangle+k}s_{\langle
j\rangle-k} \big[\zeta_j^k\big]^{(n+1)}
=&\Big[-h^2\sum\limits_{m\geq 2}\frac{g^{(m)}(0)}{m!}
\sum\limits_{k^1+\cdots+k^m=k}\sum\limits_{j_1+\cdots+j_m\equiv j\
\textmd{mod}\ 2M}'\\
& \big(\zeta_{j_1}^{k^1}\cdot\ldots\cdot
\zeta_{j_m}^{k^m}\big)-\big(2\textmd{i}  \epsilon
h^2\omega_j\csc(h\omega_j) s_{2k}
\dot{\zeta}_j^k+\cdots\big)\Big]^{(n)}.
\end{aligned}
\end{equation}
For the initial values  \eqref{initial pl}, the iteration procedure
becomes
\begin{equation}\label{pica initial pl}
\begin{aligned}\big[\zeta_j^{\langle
j\rangle}(0)+\zeta_j^{-\langle
j\rangle}(0)\big]^{(n+1)}=&\big[q_j^0-\sum\limits_{k\neq \pm \langle
j\rangle} \zeta_j^{k}(0)\big]^{(n)},\\
\textmd{i} \omega_j\big[\zeta_j^{\langle
j\rangle}(0)-\zeta_j^{-\langle
j\rangle}(0)\big]^{(n+1)}=&\big[p_j^0-\sum\limits_{k\neq \pm \langle
j\rangle} \frac{1}{h}L_1^{k}
\zeta_j^{k}(0)+2e^{\textmd{i}(2c_1-1)h\omega_j} \epsilon
h^2\omega_j\csc(h\omega_j)\dot{\zeta}_j^{\pm\langle
j\rangle}(0)+\cdots\big]^{(n)}.
\end{aligned}
\end{equation}

Following \cite{Cohen08-1}, we  assume  that $\norm{k}\leq K:= 2N$
and  $\norm{k^i}\leq K$ for $i = 1,\ldots ,m$ in these iterations.
Each iteration step involves an initial value problem of first-order
ODEs for $\zeta_j^{\pm \langle j\rangle}$ (for $|j|\leq M$) and
algebraic equations for $\zeta_j^{k}$  with $k \neq\pm \langle
j\rangle$. We choose the starting iterates ($n = 0$) as
$\zeta_j^{k}(\tau)=0$   for $k \neq\pm \langle j\rangle$, and
$\zeta_j^{\pm\langle j\rangle}(\tau)=\zeta_j^{\pm\langle
j\rangle}(0)$, where $\zeta_j^{\pm\langle j\rangle}(0)$  are
determined by the above formula. We remark that $q_{-j}^0 = q_{j}^0$
  for real initial data, and  the above
iteration implies $\big[\zeta_{-j}^{-k}\big]^n
=\overline{\big[\zeta_{-j}^{-k}\big]^n}$  for all iterates $n$ and
all $j, k$.

\subsection{Rescaling and estimation of the nonlinear terms}\label{subsec rescal}
Following  Sect. 3.5 of \cite{Cohen08} and Sect. 6.3 of
\cite{Cohen08-1}, this subsection considers a more convenient
rescaling
\begin{equation*}
\begin{aligned}
c\zeta_{j}^{k}=\frac{\omega^{|k|}}{\epsilon^{[[k]]}}\zeta_j^k,\ \
c\zeta^{k}=\big(c\zeta_{j}^{k}\big)_{|j|\leq
M}=\frac{\omega^{|k|}}{\epsilon^{[[k]]}}\zeta^k
\end{aligned}
\end{equation*}
 in the space $\mathrm{H}^s = (H^s)^{\mathcal{K}} = \{c\zeta =
(c\zeta^{k})_{k\in \mathcal{K}}: c\zeta^{k}\in H^s \}.$ The norm of
this space is chosen as $|||c\zeta|||_s^2=\sum\limits_{k\in
\mathcal{K}} \norm{c\zeta^{k}}_s^2$ and   the superscripts $k$ are
in the set $\mathcal{K}=\{k=(k_l)_{l=0}^M\ \textmd{with integers} \
k_l:\ \norm{k}\leq K\}$ with $K = 2N.$

In order to express the non-linearity in \eqref{ljkqp} in these
rescaled variables, we define the nonlinear function
$\textbf{f}=(f_j^k)$ by
\begin{equation*}
\begin{aligned}f_j^k\big(c\zeta(\tau)\big) =&\frac{\omega^{|k|}}{\epsilon^{[[k]]}}\sum\limits_{m= 2}^N\frac{g^{(m)}(0)}{m!}
\sum\limits_{k^1+\cdots+k^m=k}\frac{\epsilon^{[[k^1]]+\cdots+[[k^m]]}}{\omega^{|k^1|+\cdots+|k^m|}}
\sum\limits_{j_1+\cdots+j_m\equiv j\ \textmd{mod}\ 2M}'
\big(c\zeta_{j_1}^{k^1}\cdot\ldots\cdot
c\zeta_{j_m}^{k^m}\big)(\tau).
\end{aligned}
\end{equation*}
 By the results given in Sect. 3.5 of  \cite{Cohen08} and Sect. 6.4 of
 \cite{Cohen08-1},
 it is easy to verify that
\begin{equation}\label{bounds f}
\begin{aligned} &\sum\limits_{k\in \mathcal{K}}
\norm{f^{k}(c\zeta)}_s^2\leq \epsilon P(|||c\zeta|||_s^2),\ \
\sum\limits_{|j|\leq M} \norm{f^{\pm\langle j\rangle
}(c\zeta)}_s^2\leq \epsilon^3 P_1(|||c\zeta|||_s^2),
\end{aligned}
\end{equation}
where $P$ and $P_1$ are polynomials with coefficients bounded
independently of $\epsilon, h,$ and $M$.

Likewise,  define   the following different rescaling
\begin{equation}\label{diff resca}
\begin{aligned}
\hat{c}\zeta_{j}^{k}=\frac{\omega^{s|k|}}{\epsilon^{[[k]]}}\zeta_j^k,\
\ \hat{c}\zeta^{k}=\big(\hat{c}\zeta_{j}^{k}\big)_{|j|\leq
M}=\frac{\omega^{s|k|}}{\epsilon^{[[k]]}}\zeta^k
\end{aligned}
\end{equation}
 in   $\mathrm{H}^1 = (H^1)^{\mathcal{K}}$ with norm
$|||\hat{c}\zeta|||_1^2=\sum\limits_{\norm{k}\leq K}
\norm{\hat{c}\zeta^{k}}_1^2$, where   $\hat{f}_j^k$ is defined as
$f_j^k$ but with $\omega^{|k|}$ replaced by $\omega^{s|k|}$. We have
 similar bounds
\begin{equation*}
\begin{aligned} &\sum\limits_{k\in \mathcal{K}}
\norm{\hat{f}^{k}(\hat{c}\zeta)}_1^2\leq \epsilon
\hat{P}(|||\hat{c}\zeta|||_1^2),\ \ \sum\limits_{|j|\leq M}
\norm{\hat{f}^{\pm\langle j\rangle }(\hat{c}\zeta)}_1^2\leq
\epsilon^3 \hat{P}_1(|||\hat{c}\zeta|||_1^2)
\end{aligned}
\end{equation*}
with other polynomials $\hat{P}$ and $\hat{P}_1$.

\subsection{Reformulation of the reverse Picard iteration}\label{subsec:ref
rev} Consider  $c\zeta = (c\zeta_j^k )\in \mathrm{H}^s$ with
$c\zeta_j^k=0$ for all $k \neq\pm \langle j\rangle$
 with $(j, k)\in\mathcal{R}_{\epsilon,h}.$
In the light of the two cases: $k=\pm \langle j\rangle$
 and $k
\neq\pm \langle j\rangle$, the components of $c\zeta$
 are split
into   $a\zeta = (a\zeta_j^k )\in \mathrm{H}^s$ and $b\zeta =
(b\zeta_j^k )\in \mathrm{H}^s$, respectively:
\begin{equation*}
\begin{aligned} &a\zeta_j^k=c\zeta_j^k\qquad \textmd{if}\ k=\pm \langle
j\rangle, \quad \textmd{and 0 else},\\
&b\zeta_j^k=c\zeta_j^k\qquad \textmd{if}\  \eqref{inequa}\
\textmd{is satisfied}, \quad \textmd{and 0 else}.
\end{aligned}
\end{equation*}
It is noted that    $a\zeta + b\zeta = c\zeta$ and
$|||a\zeta|||_s^2+|||b\zeta|||_s^2=|||c\zeta|||_s^2$.

Define differential operators $A, B$  respectively as
\begin{equation*}
\begin{aligned} (Aa\zeta)_j^{\pm \langle j\rangle}=&\frac{1}{\pm2 \textmd{i}
\epsilon h^2   \omega_j}\big(\epsilon^2 h^3 \omega_j\cot(h\omega_j)
a\ddot{\zeta}_j^{\pm  \langle
j\rangle}+\cdots\big), \\
(Bb\zeta)_j^{k}=&\frac{1}{4h\omega_j\csc(h\omega_j) s_{\langle
j\rangle+k}s_{\langle j\rangle-k}}\big(2\textmd{i} \epsilon
h^2\omega_j\csc(h\omega_j) s_{2k} b\dot{\zeta}_j^k +\cdots\big)
\\
&\qquad \qquad \qquad \qquad\qquad \qquad \textmd{for} \  (j, k) \
\textmd{satisfying }\eqref{inequa}.
\end{aligned}
\end{equation*}
According to the nonlinear  function $\textbf{f}$ of the preceding
subsection, we define the functions $F =(F ^k_ j)$ and $G = (G ^k_
j)$ with non-vanishing entries for $(j, k)$ satisfying
\eqref{inequa}:
\begin{equation*}
\begin{aligned} &F_j^{\pm \langle
j\rangle}(a\zeta,b\zeta)=\frac{1}{\pm \textmd{i}\epsilon}f_j^{\pm
\langle j\rangle}\big(c\zeta\big),\ \
  G_j^{k}(a\zeta,b\zeta)=-\frac{h^2(\omega_j+|k\cdot \omega|)}{4s_{\langle j\rangle+k}s_{\langle
j\rangle-k}}f_j^{k}\big(c\zeta\big).
\end{aligned}
\end{equation*}
Furthermore, let
\begin{equation*}
\begin{aligned} &(\Omega \zeta)_j^k=(\omega_j+|k\cdot \omega|) \zeta_j^k,\quad  (\Lambda \zeta)_j^k=\sinc (h\omega_j) \zeta_j^k.
\end{aligned}
\end{equation*}
The iterations \eqref{Pic ite j} and \eqref{Pic ite notj} then have
the form
\begin{equation}\label{Abstruct Pic ite}
\begin{aligned} &a\dot{\zeta}^{(n+1)}=\Omega^{-1}F(a\zeta^{(n)},b\zeta^{(n)})-Aa\zeta^{(n)}, \ \
 b\zeta^{(n+1)}=\Omega^{-1}\Lambda G(a\zeta^{(n)},b\zeta^{(n)})-Bb\zeta^{(n)}.
\end{aligned}
\end{equation}

From the second formula of \eqref{bounds f}, it follows that
$|||F|||_s\leq C\epsilon^{1/2}$. Moreover,  according to
\eqref{further-non-res cond}, we have $|||\Omega^{-1}F|||_s\leq C$.
 By \eqref{inequa},   one gets similar    bounds $|||G|||_s\leq C$,
  which is  valid  uniformly in $\epsilon, h, M$ on bounded subsets of
$\mathrm{H}^s$. For the derivatives of $F$ and $G$, the analogous
bounds hold. Using the conditions  \eqref{further-non-res cond} and
\eqref{inequa}, we obtain
\begin{equation*}
\begin{aligned} &\abs{\frac{\epsilon^2 h^3 \omega_j\cot(h\omega_j)}{\pm2 \textmd{i}
\epsilon h^2   \omega_j} }=\abs{\frac{\epsilon  h
 \cos(h\omega_j)}{2\sin(h\omega_j)} }\leq C\abs{\frac{\epsilon  h
  }{h\epsilon^{\frac{1}{2}}} }=  C\epsilon^{\frac{1}{2}},\\
&\abs{\frac{2\textmd{i}  \epsilon h^2\omega_j\csc(h\omega_j)
s_{2k}}{4h\omega_j\csc(h\omega_j) s_{\langle j\rangle+k}s_{\langle
j\rangle-k}}}=\frac{  \epsilon h \abs{\sin(h(k
\cdot\omega))}}{2\abs{s_{\langle j\rangle+k}s_{\langle
j\rangle-k}}}\leq C\frac{  \epsilon h^2 \abs{k
\cdot\omega}}{2\epsilon^{1/2}h^2(
 \omega_j+|k \cdot \omega|)}\leq  C\epsilon^{\frac{1}{2}}.
\end{aligned}
\end{equation*}
Therefore, the operators $A$ and $B$ are bounded as:
\begin{equation}\label{AB bound}
\begin{aligned} &|||(Aa\zeta)(\tau)|||_s\leq C
\sum\limits_{l=2}^Nh^{l-2}\epsilon^{l-3/2}|||\frac{d^l}{d\tau^l}(a\zeta)(\tau)|||_s,\\
&|||(Bb\zeta)(\tau)|||_s\leq C
\sum\limits_{l=1}^Nh^{l-1}\epsilon^{l-1/2}|||\frac{d^l}{d\tau^l}(b\zeta)(\tau)|||_s.
\end{aligned}
\end{equation}

By the definitions of  $a\zeta$, the initial value condition
\eqref{pica initial pl} can be rewritten as
\begin{equation}\label{abs ini con}
\begin{aligned} &a\zeta^{(n+1)}(0)=v+Pb\zeta^{(n)}(0)+Qa\zeta^{(n)}(0),
\end{aligned}
\end{equation}
where $v$ is defined by:
\begin{equation*}
\begin{aligned} & v_j^{\pm\langle j\rangle
}=\frac{\omega_j}{\epsilon}\Big(\frac{1}{2}q_j^0\mp\frac{\textmd{i}}{2\omega_j}p_j^0\Big).
\end{aligned}
\end{equation*}
(From \eqref{initi cond}, it follows that   $v$ is bounded in
$\mathrm{H}^s$.) The operators $P$ and $Q$   are defined by
\begin{equation*}
\begin{aligned}
(Pb\zeta)_j^{\pm\langle j\rangle}(0)=&-\frac{\omega_j}{2\epsilon
}\sum\limits_{k\neq \pm\langle j\rangle}\big(1\pm \frac{1}{\textmd{i} h \omega_j}L_1^k\big)\frac{\epsilon^{[[k]]}}{\omega^{|k|}}b\zeta_j^k(0),\\
 (Qa\zeta)_j^{\pm\langle
j\rangle}(0)=&\pm\frac{1}{\textmd{i}   \omega_j}
\Big(2e^{\textmd{i}(2c_1-1)h\omega_j} \epsilon
h^2\omega_j\csc(h\omega_j)a\dot{\zeta}_j^{\pm\langle
j\rangle}(0)+\cdots\Big).
\end{aligned}
\end{equation*}
The bound of $Q$ is obvious
\begin{equation}\label{Q bound}
\begin{aligned}
&|||(Qa\zeta)(0)|||_s\leq C
\sum\limits_{l=1}^Nh^{l}\epsilon^{l-1/2}|||\frac{d^l}{d\tau^l}(a\zeta)(0)|||_s.
\end{aligned}
\end{equation}
For the bound of $P$, we have
\begin{equation*}
\begin{aligned}
|||(Pb\zeta)(0)|||^2_s=&\sum\limits_{\norm{\langle j\rangle} \leq K}
\sum\limits_{|j|\leq M}''\omega_j^{2s}
\abs{\frac{\omega_j}{2\epsilon }\sum\limits_{k\neq \pm\langle
j\rangle}\big(1\pm \frac{1}{\textmd{i} h
\omega_j}L_1^k\big)\frac{\epsilon^{[[k]]}}{\omega^{|k|}}b\zeta_j^k(0)}^2\\
\leq&C\frac{1}{4\epsilon^2 } \sum\limits_{|j|\leq
M}''\omega_j^{2s+2} \sum\limits_{k\neq \pm\langle
j\rangle}\abs{\frac{\epsilon^{[[k]]}}{\omega^{|k|}}b\zeta_j^k(0)}^2\\
\leq&C\frac{1}{4\epsilon^2 } \sum\limits_{|j|\leq
M}''\omega_j^{2s+2}\sum\limits_{k\neq \pm\langle j\rangle}
\frac{\epsilon^{2[[k]]}}{\omega^{2|k|}}  \sum\limits_{k\neq
\pm\langle j\rangle}\abs{b\zeta_j^k(0)}^2\\
\leq&C \sum\limits_{k\neq \pm\langle j\rangle}
\frac{\epsilon^{2[[k]]}}{4\epsilon^2 }  \omega^{-2|k|}
|||(b\zeta)(0)|||^2_{s+1}\\
\leq&C  |||(b\zeta)(0)|||^2_{s+1},
\end{aligned}
\end{equation*}
where we used  the Cauchy-Schwarz inequality as well as  the bound
(see Lemma 2 of \cite{Cohen08}) $\sum\limits_{\norm{k}\leq
K}\omega^{-2|k|}\leq C.$
  The starting iterates of \eqref{abs ini con}
are chosen as $a\zeta^{(0)}(\tau ) = v$ and $b\zeta^{(0)}(\tau ) =
0$.

\subsection{Bounds of the coefficient functions}\label{subsec: bounds}
By  the non-resonance conditions \eqref{inequa}
 and \eqref{further-non-res cond}, it is obtained by induction that
the iterates $a\zeta^{(n)},\ b\zeta^{(n)}$  and their derivatives
with respect to the slow time $\tau = \epsilon t $ are bounded in
$\mathrm{H}^s$  for $0 \leq \tau\leq 1$ and $n \leq 4N$. The
$(4N)$-th iterates satisfy
\begin{equation}\label{Bounds func zeta}
\begin{array}{ll}   |||a\zeta(0)|||_s\leq C,\ \ &|||\Omega a\dot{\zeta}(\tau)|||_s\leq
C\epsilon^{1/2},\\
|||\Lambda^{-1}a\dot{\zeta}(\tau)|||_s\leq C,\ \
&|||\Lambda^{-1}\Omega b \zeta (\tau)|||_s\leq C,
\end{array}
\end{equation}
where $C$ is independent of $\epsilon, h, M$, but depends on $N$.
The analogous bounds for higher derivatives of these functions with
respect to $\tau$ can also be obtained. These bounds
  together imply the bounds \eqref{bound modu func}.
According to these bounds, it is also true that $|||c\zeta
(\tau)-a\zeta(0)|||_{s+1}\leq C$. By Sect. 3.7 of \cite{Cohen08},
Sect. 6.6 of \cite{Cohen08-1} and \eqref{AB bound}, it can be
confirmed that the bound \eqref{bound TMFE} is true.

From \eqref{bounds f} and \eqref{Abstruct Pic ite}, it follows that
$\big(\sum\limits_{\norm{k}=1}\norm{(\Lambda^{-1}\Omega b
\zeta)^k}_s^2\big)^{1/2}\leq C \epsilon$ for $b \zeta=(b
\zeta)^{(4N)}$. Moreover, in the light of \eqref{another-non-res
cond}, one arrives that $$ \sum\limits_{|j|\leq
M}\sum\limits_{j_1+j_2=j}\sum\limits_{k=\pm\langle
j_1\rangle\pm\langle j_2\rangle}
 \omega_j^{2(s+1)}|b\zeta_j^k|^2 \leq C \epsilon.$$
  These bounds  as well as   \eqref{abs ini con} yield
the   formula of \eqref{jth TMFE}.

The same bounds can be obtained for the alternative scaling
\eqref{diff resca}:
\begin{equation}\label{Bounds  zeta sca}
\begin{aligned} & |||\hat{a}\zeta(0)|||_1\leq C,\ \ |||\Omega \hat{a}\dot{\zeta}(\tau)|||_1\leq
C\epsilon^{1/2},\ \  |||\Lambda^{-1}\Omega \hat{b} \zeta
(\tau)|||_1\leq C .
\end{aligned}
\end{equation}
Moreover,  the following bounds hold for this scaling:
\begin{equation}\label{Bounds  eta sca}
\begin{aligned}
&\big(\sum\limits_{\norm{k}=1}\norm{(\Lambda^{-1}\Omega \hat{b}
\zeta)^k}_1^2\big)^{1/2}\leq C \epsilon.
\end{aligned}
\end{equation}

\subsection{Defect}\label{subsec:def}
The defect in \eqref{methods} can be expressed  as
\begin{equation*}
\begin{aligned}
&\delta_j(t)=\sum\limits_{\norm{k}\leq NK}d^k(\epsilon t)
e^{\textmd{i}(k\cdot\omega) t}+R_{N+1}(t),
\end{aligned}
\end{equation*}
where
\begin{equation}\label{djk}
\begin{aligned}
d_j^k=&\frac{1}{h^2}\tilde{L}^k \zeta_j^k+\sum\limits_{m=
2}^N\frac{g^{(m)}(0)}{m!}
\sum\limits_{k^1+\cdots+k^m=k}\sum\limits_{j_1+\cdots+j_m\equiv j\
\textmd{mod}\ 2M}'  \big(\zeta_{j_1}^{k^1}\cdot\ldots\cdot
\zeta_{j_m}^{k^m}\big).
\end{aligned}
\end{equation}
Here $\zeta^k_j=(\zeta^k_j)^{(4N)}$  are obtained after $4N$
iterations of the procedure in Subsection \ref{subsec interation}.
We consider   $\norm{k}\leq NK$, and $\zeta_j^k=0$ for
$\norm{k}>K:=2N$. We use  $\tilde{L}^k$ to denote  the truncation of
the  operator $L^k$ after the $\epsilon^N$ term. The function
$R_{N+1}$  contains the remainder terms of the Taylor expansion of
$f$ after $N$ terms. We obtain
 $\norm{R_{N+1}}_{s+1} \leq
C\epsilon^{N+1}$  by the bound \eqref{bound TMFE} for the remainder
in the Taylor expansion of $f$ and the estimate \eqref{Bounds func
zeta}  for the $(N + 1)$-th derivative for $\zeta_j^k$.

 According to the analysis given in Sect. 3.8 of \cite{Cohen08} and Sect. 6.7 of  \cite{Cohen08-1},
 it is   obtained that
\begin{equation*}
\begin{aligned}
&\norm{\sum\limits_{\norm{k}\leq NK}d^k(\epsilon t)
e^{\textmd{i}(k\cdot\omega) t}}_{s}^2\leq C
\sum\limits_{\norm{k}\leq NK}\norm{\omega^{|k|}d^k(\epsilon
t)}_{s}^2.
\end{aligned}
\end{equation*}
In what follows,  we prove the bounds
\begin{equation*}
\begin{aligned}
& \sum\limits_{\norm{k}\leq NK}\norm{\omega^{|k|}d^k(\epsilon
t)}_{s}^2\leq C\epsilon^{2(N+1)}
\end{aligned}
\end{equation*}
for three cases: truncated, near-resonant and non-resonant modes,
which will be shown by the following three parts.
\subsubsection{Truncated  mode}
 It is
known that we set $\zeta_j^k=0$ for $\norm{k}>K:=2N$ (truncated
modes) and for $(j, k) \in\mathcal{R}_{\epsilon,h}$ (near-resonance
modes). Thence for these  both cases, the defect has the form
\begin{equation*}
\begin{aligned}
d_j^k=&\sum\limits_{m= 2}^N\frac{g^{(m)}(0)}{m!}
\sum\limits_{k^1+\cdots+k^m=k}\sum\limits_{j_1+\cdots+j_m\equiv j\
\textmd{mod}\ 2M}' \big(\zeta_{j_1}^{k^1}\cdot\ldots\cdot
\zeta_{j_m}^{k^m}\big).
\end{aligned}
\end{equation*}
For truncated modes we express the defect as
\begin{equation*}
\begin{aligned}
&d_j^k=\epsilon^{[[k]]}\omega^{-|k|} f_j^{k}\big(c\zeta\big).
\end{aligned}
\end{equation*}
On the basis of  \eqref{bounds f} and \eqref{Bounds func zeta}
   with $2N$ instead  of $K$, we have the
bound $|||f|||_s^2\leq C \epsilon$, which yields
\begin{equation*}
\begin{aligned}
&\sum\limits_{\norm{k}>K}\sum\limits_{|j|\leq
M}'\omega_j^{2s}|\omega^{|k|}d_j^k|^2\leq
\sum\limits_{\norm{k}>K}\sum\limits_{|j|\leq
M}'\omega_j^{2s}|f_j^k|^2\epsilon^{2[[k]]}\leq C\epsilon^{2(N+1)},
\end{aligned}
\end{equation*}
where we have used  the fact that $2[[k]]= \norm{k} + 1 \geq K + 2 =
2(N + 1).$

\subsubsection{Near-resonant mode} For the near-resonant modes we express the  defect by the rescaling
\eqref{diff resca} as follows
\begin{equation*}
\begin{aligned}
d_j^k=& \epsilon^{[[k]]}\omega^{-s|k|}
\hat{f}_j^{k}\big(\hat{c}\zeta\big).
\end{aligned}
\end{equation*}
Therefore, it is obtained that
\begin{equation*}
\begin{aligned}
 \sum\limits_{(j,k)\in
\mathcal{R}_{\epsilon,h}}\omega_j^{2s}|\omega^{|k|}d_j^k|^2&=\sum\limits_{(j,k)\in
\mathcal{R}_{\epsilon,h}}\frac{\omega_j^{2(s-1)}}{\omega^{2(s-1)|k|}}\epsilon^{2[[k]]}\omega_j^2|\hat{f}_j^k|^2\\
&\leq C\sup_{(j,k)\in
\mathcal{R}_{\epsilon,h}}\frac{\omega_j^{2(s-1)}}{\omega^{2(s-1)|k|}}\epsilon^{2[[k]]+1}\leq
C\epsilon^{2(N+1)}
\end{aligned}
\end{equation*}
by considering  $|||\hat{f}|||_1^2\leq C \epsilon$ and the
non-resonance condition  \eqref{non-resonance cond}.

\subsubsection{Non-resonant mode} Now consider the
non-resonant mode, which means that $\norm{k}\leq K$ and that $(j,
k)$ satisfies the non-resonance condition \eqref{inequa}. For this
case, the defect is formulated   in the scaled variables of
Subsection \ref{subsec rescal}  as
\begin{equation*}
\begin{aligned}
\omega^{|k|}d_j^k=& \epsilon^{[[k]]}
\Big(\frac{1}{h^2\bar{b}_1(h\omega_j)}\tilde{L}_j^k
c\zeta_j^k+f_j^{k}\big(c\zeta \big)\Big).
\end{aligned}
\end{equation*}
Split them into  $k={\pm\langle j\rangle}$ and $k\neq{\pm\langle
j\rangle}$ and we obtain
\begin{equation*}
\begin{array}{ll}
\omega_jd_j^{\pm\langle j\rangle}= \epsilon \Big( \frac{\pm2
\textmd{i} \epsilon h^2   \omega_j}{h^2}
\big(a\dot{\zeta}_j^{\pm\langle j\rangle}+(Aa\zeta)_j^{\pm\langle
j\rangle}\big)+f_j^{\pm\langle j\rangle}\big(c\zeta \big)\Big),\\
\omega^{|k|}d_j^k=\epsilon^{[[k]]}
\Big(\frac{4h\omega_j\csc(h\omega_j) s_{\langle
j\rangle+k}s_{\langle
j\rangle-k}}{h^2}\big(b\zeta_j^k+(Bb\zeta)_j^k\big)+f_j^{k}\big(c\zeta
\big)\Big).
\end{array}
\end{equation*}
We remark that the functions here are actually the $4N$-th iterates
of the iteration in Subsection   \ref{subsec interation}. Expressing
$f_j^{\pm\langle j\rangle}\big(c\zeta(t\epsilon)\big)$ and
$f_j^{k}\big(c\zeta(t\epsilon)\big)$ in terms of $F$ and $G$ and
inserting them from \eqref{Abstruct Pic ite} into this defect, we
obtain
\begin{equation*}
\begin{array}{ll}
\omega_jd_j^{\pm\langle j\rangle}= 2 \omega_j \alpha_j^{\pm\langle
j\rangle}\big( \big[a\dot{\zeta}_j^{\pm\langle
j\rangle}\big]^{(4N)}-\big[a\dot{\zeta}_j^{\pm\langle
j\rangle}\big]^{(4N+1)}\big), \ \ \ & \alpha_j^{\pm\langle
j\rangle}=\pm  \textmd{i}\epsilon^2,\\
\omega^{|k|}d_j^k= \beta_j^k\big([b\zeta_j^{k}]^{(4N)}-
[b\zeta_j^{k} ]^{(4N+1)}\big), \ \ \ &
\beta_j^{k}=\epsilon^{[[k]]}\frac{4h\omega_j\csc(h\omega_j)
s_{\langle j\rangle+k}s_{\langle j\rangle-k}}{h^2}.
\end{array}
\end{equation*}
It is noted that these expressions are very similar to those in
Sect. 6.9 of \cite{Cohen08-1}. Thus, following the analysis given in
 \cite{Cohen08-1}, for $\tau\leq1,$ we have
\begin{equation}\label{right hand esti non}
\begin{aligned}
&\Big( \sum\limits_{\norm{k}\leq
K}\norm{\omega^{|k|}d^k(\tau)}_{s}^2\Big)^{1/2}\leq C\epsilon^{N+1}.
\end{aligned}
\end{equation}

Then the defect   $\delta(t)$ becomes $
 \norm{\delta(t)}_{s}\leq C\epsilon^{N+1}$   for $t\leq
\epsilon^{-1}.$ Considering  the defect in the initial conditions
\eqref{initial pl} and \eqref{initial mi}, one  has $$
\norm{q^0-\tilde{q}(0)}_{s+1}+\norm{p^0-\tilde{p}(0)}_{s}\leq
C\epsilon^{N+1}.$$ For the alternative scaling \eqref{diff resca},
it is true that \begin{equation}\label{alter esti non}
\begin{aligned}
&\Big( \sum\limits_{\norm{k}\leq
K}\norm{\omega^{s|k|}d^k(\tau)}_{1}^2\Big)^{1/2}\leq
C\epsilon^{N+1}.
\end{aligned}
\end{equation}

\subsection{Remainders}\label{subsec:rem}
Let $\Delta q^n=\tilde{q}(t_n)-q^n,\ \ \Delta
p^n=\tilde{p}(t_n)-p^n$. We then have
\begin{equation*}
\begin{aligned}
\left(
  \begin{array}{c}
    \Delta q^{n+1} \\
    \Omega^{-1}\Delta p^{n+1} \\
  \end{array}
\right)=\left(
          \begin{array}{cc}
            \cos(h\Omega) & \sin(h\Omega) \\
            -\sin(h\Omega) & \cos(h\Omega) \\
          \end{array}
        \right)\left(
  \begin{array}{c}
    \Delta q^n \\
    \Omega^{-1}\Delta p^n \\
  \end{array}
\right)+h\left(
          \begin{array}{c}
           h\Omega\bar{b}_1(h\Omega)\Omega^{-1}( \Delta f +\delta)\\
            b_1(h\Omega) \Omega^{-1} (\Delta f + \delta) \\
          \end{array}
        \right),
\end{aligned}
\end{equation*}
where $\Delta f=\big(f(\tilde{q}_{h}(t_n)-f(Q^{n+c_1})\big).$ It
follows from  the Lipschitz bound (Sect. 4.2 in \cite{Hairer08} and
Sect. 6.10 in \cite{Cohen08-1}) that
\begin{equation*}
\begin{aligned}
\norm{\Omega^{-1}\Delta f}_{s+1}=\norm{\Delta f}_{s}\leq \epsilon
\norm{\tilde{q}_{h}(t_n)- Q^{n+c_1 }}_{s}\leq \epsilon (\norm{\Delta
q^{n}}_{s}+\norm{\Delta p^{n}}_{s-1}).
\end{aligned}
\end{equation*}
where the first formula of \eqref{methods} and \eqref{MFE-2} are
considered here. Moreover, we have
$\norm{\Omega^{-1}\delta(t)}_{s+1}=\norm{\delta(t)}_{s}\leq
C\epsilon^{N+1}$. Therefore, we obtain
\begin{equation*}
\begin{aligned}
\norm{\left(
  \begin{array}{c}
    \Delta q^{n+1} \\
    \Omega^{-1}\Delta p^{n+1} \\
  \end{array}
\right) }_{s+1}\leq\norm{ \left(
  \begin{array}{c}
    \Delta q^{n} \\
    \Omega^{-1}\Delta p^{n} \\
  \end{array}
\right)}_{s+1}+h\Big(C\epsilon \norm{ \Delta q^{n}}_{s+1}+C\epsilon
\norm{\Delta p^{n}}_{s+1}+C\epsilon^{N+1}\Big).
\end{aligned}
\end{equation*}
Solving this inequality leads to $ \norm{ \Delta
q^{n}}_{s+1}+\norm{\Omega^{-1} \Delta p^{n}}_{s+1}\leq
C(1+t_n)\epsilon^{N+1} $ for $t_n\leq\epsilon^{-1}.$ This proves
\eqref{error MFE} and completes the proof of Theorem \ref{MFE thm}.

\subsection{Three almost-invariants} \label{subsec:alm inv}

 In
this section, we will    show that the modulated Fourier expansion
\eqref{MFE-ERKN} has three almost-invariants which are closed to the
actions, the momentum and the total energy of \eqref{prob}.


According to the proof of   Theorem \ref{MFE thm},  the defect
formula \eqref{djk} can be reformulated as
\begin{equation}\label{djk pote}
\begin{aligned}
&\frac{1}{h^2}\tilde{L}^k
\zeta_j^k+\nabla_{-j}^{-k}\mathcal{U}(\zeta)=d_j^k,
\end{aligned}
\end{equation}
where $\nabla_{-j}^{-k}\mathcal{U}(y)$ is the partial derivative
with respect to $y_{-j}^{-k}$ of the extended potential
(\cite{Cohen08-1,Hairer08})
\begin{equation*}
\begin{aligned}
&\mathcal{U}(\zeta)=\sum\limits_{l=-N}^N\mathcal{U}_l(\zeta),\\
&\mathcal{U}_l(\zeta)=\sum\limits_{m= 2}^N
\frac{U^{(m+1)}(0)}{(m+1)!}
\sum\limits_{k^1+\cdots+k^{m+1}=0}\sum\limits_{j_1+\cdots+j_{m+1}=2Ml}'
\big(\zeta_{j_1}^{k^1}\cdot\ldots\cdot
\zeta_{j_{m+1}}^{k^{m+1}}\big).
\end{aligned}
\end{equation*}
Here $U$ is the potential appearing in \eqref{H} and $\norm{k^i}\leq
2N$ and $| j_i| \leq M$.

 Define
$S_{\mu}(\theta)y=\big(\mathrm{e}^{\mathrm{i}(k \cdot \mu)
\theta}y_j^k\big)_{| j| \leq M, \norm{k}\leq K}$ and
$T(\theta)y=\big(\mathrm{e}^{\mathrm{i}j \theta}y_j^k\big)_{| j|
\leq M, \norm{k}\leq K},$ where $\mu = (\mu_l)_{l\geq0}$ is an
arbitrary real sequence  for $\theta\in R.$ It is easy to see that
the  summation  in the definition of $\mathcal{U}$ is over $k^1
+\ldots+k^{m+1} = 0$ and that in $\mathcal{U}_0$ over
$j_1+\cdots+j_{m+1} = 0$, which yield  $
\mathcal{U}(S_{\mu}(\theta)y)=\mathcal{U}(y)$ and
$\mathcal{U}_0(T(\theta)y)=\mathcal{U}_0(y)$ for $\theta\in
\mathbb{R}$ (see \cite{Cohen08-1}). Therefore, we have
\begin{equation}\label{du}
\begin{aligned}
0=\frac{d}{d\theta}\mid_{\theta=0}\mathcal{U}(S_{\mu}(\theta)\zeta),\
\ \ \ 0=\frac{d}{d\theta}\mid_{\theta=0}
\mathcal{U}_0(T(\theta)\zeta).
\end{aligned}
\end{equation}

By the first formula of \eqref{du}, we obtain
\begin{equation*}
\begin{aligned}
&0=\frac{d}{d\theta}\mid_{\theta=0}\mathcal{U}(S_{\mu}(\theta)\zeta)
= \sum\limits_{\norm{k}\leq K}\sum\limits_{|j|\leq M}' \mathrm{i}(k
\cdot \mu)   \zeta_{-j}^{-k} \nabla_{-j}^{-k}\mathcal{U}(\zeta) .
\end{aligned}
\end{equation*}
 According to \eqref{djk pote} and the above formula,
we have
\begin{equation}
\begin{aligned}
& \sum\limits_{\norm{k}\leq K}\sum\limits_{|j|\leq M}' \mathrm{i}(k
\cdot \mu)   \zeta_{-j}^{-k} \frac{1}{h^2}\tilde{L}^k \zeta_j^k =
\sum\limits_{\norm{k}\leq K}\sum\limits_{|j|\leq M}' \mathrm{i}(k
\cdot \mu)   \zeta_{-j}^{-k} d_j^k.
\end{aligned}
\label{almost 1-2}%
\end{equation}
By the expansion \eqref{L expansion} of the operator $L^k$ and the
analysis given in Sect. 7.3 of \cite{Cohen08-1}, we know that
$\zeta_{-j}^{-k} \tilde{L}_j^k \zeta_j^k$  is a total derivative.
 Therefore, the left-hand side of \eqref{almost 1-2} is a total derivative of
function $\epsilon \mathcal{J}_{\mu}[\zeta]$ which depends on
$\zeta(\tau), \eta(\tau)$ and their up to $N-1$th order derivatives.
In this way, \eqref{almost 1-2} becomes
\begin{equation}
\begin{aligned}
& -\epsilon\frac{d}{d \tau} \mathcal{J}_{\mu}[\zeta]=
\sum\limits_{\norm{k}\leq K}\sum\limits_{|j|\leq M}' \mathrm{i}(k
\cdot \mu)   \zeta_{-j}^{-k} d_j^k.
\end{aligned}
\label{almost 1-3}%
\end{equation}

By the second formula of \eqref{du} and in a similar way, we obtain
\begin{equation}
\begin{aligned}
& \sum\limits_{\norm{k}\leq K}\sum\limits_{|j|\leq M}' \mathrm{i}j
  \frac{1}{h^2}\zeta_{-j}^{-k}
 \tilde{L}^k
\zeta_j^k  =   \sum\limits_{\norm{k}\leq K}\sum\limits_{|j|\leq M}'
\mathrm{i}j  \zeta_{-j}^{-k} \Big(d_j^k-
\sum\limits_{l\neq0}\nabla_{-j}^{-k} \mathcal{U}_l(\zeta) \Big) .
\end{aligned}
\label{almost 2-2}%
\end{equation}
The left-hand side of this can be written as a total derivative of
function  $\epsilon \mathcal{K}[\zeta](\tau)$ and \eqref{almost 2-2}
becomes
\begin{equation}
\begin{aligned}
& -\epsilon\frac{d}{d \tau} \mathcal{K}[\zeta]  =
\sum\limits_{\norm{k}\leq K}\sum\limits_{|j|\leq M}' \mathrm{i}j
\zeta_{-j}^{-k}\Big(d_j^k- \sum\limits_{l\neq0}\nabla_{-j}^{-k}
\mathcal{U}_l(\zeta) \Big).
\end{aligned}
\label{almost 2-3}%
\end{equation}
Another almost-invariant is obtained by considering
\begin{equation*}
\begin{aligned}
& \epsilon\frac{d}{d\tau} \mathcal{U}(\zeta)
=\sum\limits_{\norm{k}\leq K}\sum\limits_{|j|\leq M}'
\dot{\zeta}_{-j}^{-k} \nabla_{-j}^{-k}\mathcal{U}(\zeta)\\
=&  \sum\limits_{\norm{k}\leq K}\sum\limits_{|j|\leq M}' \big
(\textmd{i}(k\cdot\omega)\zeta_{-j}^{-k}+
 \epsilon\dot{\zeta}_{-j}^{-k} \big)
\big (-\frac{1}{h^2}\tilde{L}^k \zeta_j^k+d_j^k\big).
\end{aligned}
\end{equation*}
Similarly to the   analysis in Sect. 7.3 of \cite{Cohen08-1}, we
know that there is a function $\epsilon \mathcal{H}[\zeta](\tau)$
such that
\begin{equation}
\begin{aligned}
-\epsilon\frac{d}{d \tau} \mathcal{H}[\zeta](\tau)=&
\sum\limits_{\norm{k}\leq K}\sum\limits_{|j|\leq M}' \big
(\textmd{i}(k\cdot\omega)\zeta_{-j}^{-k}+
 \epsilon\dot{\zeta}_{-j}^{-k} \big) d_j^k.
\end{aligned}
\label{almost 3-2}%
\end{equation}

Consider the special case of $\mu=\langle l\rangle$ for
$\mathcal{J}_{\mu}$ and the following result is obtained on noticing
the smallness of the right-hand sides in \eqref{almost 1-3},
\eqref{almost 2-3}  and \eqref{almost 3-2}.

\begin{theo}\label{invariant12 thm}
Under the conditions of Theorem \ref{MFE thm},  for $\tau\leq 1,$ it
is true that
\begin{equation*}
\begin{aligned}
&\sum\limits_{l=1}^M \omega_l^{2s+1} \left|\frac{d}{d \tau}
\mathcal{J}_{l}[\zeta](\tau)\right|\leq C\epsilon^{N+1},\\
&\left|\frac{d}{d \tau} \mathcal{K}[\zeta](\tau)\right|\leq
C(\epsilon^{N+1}+\epsilon^{2}M^{-s+1}),\\ & \left|\frac{d}{d \tau}
\mathcal{H}[\zeta](\tau)\right|\leq C\epsilon^{N+1}.
\end{aligned}
\end{equation*}
\end{theo}
\begin{proof}  It follows from  the bounds \eqref{Bounds  zeta sca} and \eqref{alter
esti non} that  the first estimate  is obtained  by using a similar
proof to Theorem 3 in \cite{Cohen08}.   Based on the bounds
\eqref{bound modu func} and \eqref{right hand esti non}, the second
estimate can be proved as in Theorem 5.2 of \cite{Hairer08}. As in
Theorem 6 of \cite{Cohen08-1}, the third estimate   is got by the
Cauchy-Schwarz inequality and the estimates \eqref{bound modu func}
and \eqref{right hand esti non}.
\end{proof}

\subsection{Relationship with the quantities of \eqref{prob}} \label{subsec:rela}
Denote the harmonic action of the numerical solution by
$$J_{l} = I_l + I_{-l} = 2I_l\quad \textmd{for} \quad 0 < l <
M,\quad J_0 = I_0,\quad J_M = I_M.$$ It can be obtained that the
almost-invariants $\mathcal{J}_{l}$, $\mathcal{H}$ and $\mathcal{K}$
are close to the harmonic action $J_{l}$, the Hamiltonian $H_M$ and
the momentum $K$, respectively.

\begin{theo}\label{Relationship thm}
Under the conditions of Theorem \ref{main theo1}, along the
numerical solution $(q^n , p^n)$ of \eqref{methods} and the
associated modulation sequence $(\zeta(\epsilon t), \eta(\epsilon
t))$, we have
\begin{equation*}
\begin{aligned}
&  \mathcal{J}_{l}[\zeta](\epsilon t_n)=J_{l}(q^n ,
p^n)+\gamma_l(t_n)\epsilon^3,\\
&\mathcal{K}[\zeta](\epsilon t_n)=K(q^n ,
p^n)+\mathcal{O}(\epsilon^3)+\mathcal{O}(\epsilon^2 M^{-s}),\\
&\mathcal{H}[\zeta](\epsilon t_n)=H_M(q^n ,
p^n)+\mathcal{O}(\epsilon^3),
\end{aligned}
\end{equation*}
where all the constants are independent of $\epsilon, M, h$, and
$n$, and  $\sum\limits_{l=0}^M \omega_l^{2s+1} \gamma_l(t_n)\leq C$
for $t_n\leq \epsilon^{-1}.$
\end{theo}
\begin{proof}We only prove the second statement of this theorem. The others can be obtained in a similar way and we
skip them for brevity.

  According to \eqref{L expansion}, \eqref{almost 2-2} and the
``magic formulas" on p. 508 of \cite{hairer2006},  it is yielded
that
\begin{equation*}
\begin{aligned}
 &\mathcal{K}[\zeta](\tau) = \sum\limits_{|j|\leq M}' j\omega_j \Big(
 |\zeta_{j}^{\langle j\rangle}|^2- |\zeta_{j}^{-\langle j\rangle}|^2\Big)+\mathcal{O}(\epsilon^3).
\end{aligned}
\end{equation*}

By \eqref{jth TMFE}, \eqref{impor-zetaeta} and the bounds of the
coefficient functions,  we get $\zeta_j^{\langle
j\rangle}=\frac{1}{2}\big(\tilde{q}_j+\frac{1}{\mathrm{i}\omega_j}\tilde{p}_j\big)+\mathcal{O}(\epsilon^2)$
and $\zeta_j^{-\langle
j\rangle}=\frac{1}{2}\big(\tilde{q}_j-\frac{1}{\mathrm{i}\omega_j}\tilde{p}_j\big)+\mathcal{O}(\epsilon^2).
$  Therefore, the scheme of $ \mathcal{K}$ becomes (see
\cite{Cohen08-1})
\begin{equation*}
\begin{aligned}
 \mathcal{K}[\zeta](\tau)
 =& \sum\limits_{|j|\leq M}'\frac{ j\omega_j }{4}\Big(
 |\tilde{q}_j+\frac{1}{\mathrm{i}\omega_j}\tilde{p}_j|^2-
 |\tilde{q}_j-\frac{1}{\mathrm{i}\omega_j}\tilde{p}_j|^2\Big)+\mathcal{O}(\epsilon^3)\\
= &  \sum\limits_{|j|\leq M}'\frac{ j\omega_j }{4}4
\frac{1}{\mathrm{i}\omega_j} \tilde{q}_{-j}\tilde{p}_j
+\mathcal{O}(\epsilon^3) \\=&
K(\tilde{q},\tilde{p})+\mathcal{O}(\epsilon^3)+\mathcal{O}(\epsilon^2
M^{-s})   \\=&
K(q^n,p^n)+\mathcal{O}(\epsilon^3)+\mathcal{O}(\epsilon^2 M^{-s}),
\end{aligned}
\end{equation*}
where the results \eqref{jth TMFE} and \eqref{error MFE} are used.
\end{proof}

\subsection{Proof of Theorem \ref{main theo1}}\label{subsec:fro to}
By the analysis given in this paper and following
\cite{Cohen08,Cohen08-1},  the statement of Theorem \ref{main theo1}
can be proved by patching together many intervals of length
$\epsilon^{-1}$.

\section{Proof of Theorem \ref{main
theo2}} \label{sec: symmetric proof} Proof of Theorem \ref{main
theo2} is very close to that of   Theorem \ref{main theo1} and we
only present the main differences in this section.

We still use the   operators defined in \eqref{new L1L2}. For the
coefficients of ERKN methods satisfying the symmetry conditions but
not the symplecticity conditions \eqref{symple cond}, the  Taylor
expansions of the operator $L^{k}$ become
\begin{equation}\label{Lll expansion}
\begin{aligned}L^{\pm\langle j\rangle}(hD)\alpha_j^{\pm\langle j\rangle}(\epsilon t)= &\pm2 \textmd{i}
\epsilon h\sin(h\omega_j/2)/\bar{b}_1(h\omega_j)
\dot{\alpha}_j^{\pm\langle j\rangle}(\epsilon t)\\ &+ \epsilon^2 h^2
 \cos(h\omega_j)\sec(h\omega_j/2)/(2\bar{b}_1(h\omega_j))\ddot{\alpha}_j^{\pm\langle
j\rangle}(\epsilon t)+\cdots,\\
 L^{k}(hD)\alpha_j^{k}(\epsilon t)=& 2\sec(h\omega_j/2)/\bar{b}_1(h\omega_j) s_{\langle j\rangle+k}s_{\langle
j\rangle-k}\alpha_j^{k}(\epsilon t)\\
&+ \textmd{i}  \epsilon
h \sec(h\omega_j/2)/\bar{b}_1(h\omega_j) s_{2k} \dot{\alpha}_j^{k}(\epsilon t)+\cdots,\\
\end{aligned}
\end{equation}
for $\abs{j}>0$ and $k\neq \pm\langle j\rangle$. A result of
modulated Fourier expansion (the same as Theorem \ref{MFE thm}) for
symmetric but not symplectic ERKN methods can also be obtained. This
result can be proved by the same way as Sects. \ref{subsec:mod
equ}-\ref{subsec:rem} with some obvious modifications brought by the
difference between \eqref{L expansion} and   \eqref{Lll expansion}.
Then three almost-invariants of the modulation system can be derived
as follows
\begin{equation*}
\begin{aligned}
&\mathcal{J}_{l}[\zeta](\tau)=  \omega_l\sigma(h \omega_l) \Big(
 |\zeta_{l}^{\langle l\rangle}|^2+ |\zeta_{l}^{-\langle
 l\rangle}|^2\Big)+\mathcal{O}(\epsilon^3),\\
 &\mathcal{K}[\zeta](\tau) = \sum\limits_{|j|\leq M}' j\omega_j\sigma(h \omega_j) \Big(
 |\zeta_{j}^{\langle j\rangle}|^2- |\zeta_{j}^{-\langle
 j\rangle}|^2\Big)+\mathcal{O}(\epsilon^3),\\
 &\mathcal{H}[\zeta](\tau) = \sum\limits_{|j|\leq M}'  \omega^2_j\sigma(h \omega_j) \Big(
 |\zeta_{j}^{\langle j\rangle}|^2+ |\zeta_{j}^{-\langle
 j\rangle}|^2\Big)+\mathcal{O}(\epsilon^3),
\end{aligned}
\end{equation*}
where $\sigma(h \omega_j)= \frac{1}{2}\textmd{sinc}\big(\frac{1}{2}h
\omega_j\big)/\bar{b}_1(h \omega_j).$ Finally Theorem \ref{main
theo2} is obtained by exploring the connections between these
almost-invariants and the modified quantities   \eqref{MMMM}.

\section{Proof of Corollary \ref{main
theo3}} \label{sec: ss proof} It is clearly that Corollary \ref{main
theo3} can be obtained immediately by Theorem \ref{main theo1} or
\ref{main theo2}. In this section, we present another proof by
establishing a relationship between symplectic and symmetric ERKN
methods and trigonometric integrators.

We first consider the following Strang splitting method
 \begin{equation*}
\begin{array}[c]{ll}
&\textmd{1. } (q_{+}^{n},p_{+}^{n}) = \Phi_{h/2,\textmd{L}}
(q^{n},p^{n}): \\
&\left(
             \begin{array}{c}
              q_{+}^{n} \\
              p_{+}^{n} \\
             \end{array}
           \right)=\left(
                                                           \begin{array}{cc}
                                                            \cos(\frac{h\Omega}{2})& \Omega^{-1}\sin(\frac{h\Omega}{2}) \\
                                                             -\Omega \sin(\frac{h\Omega}{2}) & \cos(\frac{h\Omega}{2}) \\
                                                           \end{array}
                                                         \right)\left(
             \begin{array}{c}
               q^{n} \\
               p^{n} \\
             \end{array}
           \right),\\
&\textmd{2. }(q_{-}^{n},p_{-}^{n}) = \Phi_{h,\textmd{NL}}
(q_{+}^{n},p_{+}^{n}):\\
&\left(
             \begin{array}{c}
              q_{-}^{n} \\
              p_{-}^{n} \\
             \end{array}
           \right)=\left(
             \begin{array}{c}
              q_{+}^{n} \\
              p_{+}^{n}+h\Upsilon(h \Omega)\tilde{g}(q_{-}^{n}) \\
             \end{array}
           \right),\\
&\textmd{3. }(q^{n+1},p^{n+1}) = \Phi_{h/2,\textmd{L}}
(q_{-}^{n},p_{-}^{n}): \\
&\left(
             \begin{array}{c}
              q^{n+1} \\
              p^{n+1} \\
             \end{array}
           \right)=\left(
                                                           \begin{array}{cc}
                                                            \cos(\frac{h\Omega}{2})& \Omega^{-1}\sin(\frac{h\Omega}{2}) \\
                                                             -\Omega \sin(\frac{h\Omega}{2}) & \cos(\frac{h\Omega}{2}) \\
                                                           \end{array}
                                                         \right)\left(
             \begin{array}{c}
               q_{-}^{n} \\
               p_{-}^{n}\\
             \end{array}
           \right),
\end{array}
\end{equation*}
 where
$\Upsilon$ is a function of $h \Omega$ which will be determined
below. Then a one-stage explicit symmetric ERKN method
\eqref{methods} with the symmetry condition \eqref{sym cond} is
denoted by $\Phi_h$, i.e., $(q^{n+1},p^{n+1}) = \Phi_h
(q^{n},p^{n}).$
 It can be straightly verified that  this ERKN method $\Phi_h$ can be expressed by the Strang splitting
method as
 \begin{equation}
\Phi_h=\Phi_{h/2,\textmd{L}} \circ \Phi_{h,\textmd{NL}} \circ
\Phi_{h/2,\textmd{L}}
  \label{connetciton1}%
\end{equation}
if and only if
 \begin{equation}
\frac{1}{2}\textmd{sinc}\big(\frac{1}{2}h \Omega\big)\Upsilon(h
\Omega)=\bar{b}_1(h \Omega),\ \ \cos\big(\frac{1}{2}h
\Omega\big)\Upsilon(h \Omega)=b_1(h \Omega).
  \label{connetciton10}%
\end{equation}
It is noted that these conditions hold under the symmetry conditions
\eqref{sym cond}. If   the symplecticity conditions \eqref{symple
cond} are considered further, one gets
 \begin{equation}
\bar{b}_1(h \Omega)=\frac{1}{2}\textmd{sinc}\big(\frac{1}{2}h
\Omega\big),\ \ b_1(h \Omega)=\cos\big(\frac{1}{2}h \Omega\big),\ \
\Upsilon(h \Omega)=1.
  \label{connetciton10}%
\end{equation}


 On the other hand,  we consider another Strang splitting
 \begin{equation*}
\hat{\Phi}_h=\Phi_{h/2,\textmd{NL}} \circ \Phi_{h,\textmd{L}} \circ
\Phi_{h/2,\textmd{NL}},
\end{equation*}
which yields a class of   trigonometric integrators
\begin{equation}\label{TI}
\begin{array}[c]{ll} &q^{n+1}=
\cos(h\Omega)q^{n}+h\textmd{sinc}(h\Omega)p^{n}+\frac{1}{2}h^2
 \textmd{sinc}(h\Omega)\Upsilon(h
\Omega)\tilde{g}(q^n),\\
 &p^{n+1}=-\Omega
 \sin(h\Omega)q^{n}+\cos(h\Omega)p^{n}+\frac{1}{2}h\big(
 \cos(h\Omega)\Upsilon(h \Omega)\tilde{g}(q^n)+\Upsilon(h
\Omega)\tilde{g}(q^{n+1})\big).
\end{array}
\end{equation}
It is noted that this is exactly a symmetric trigonometric
integrator (a form of (XIII.2.7)--(XIII.2.8) given on p.481 of
\cite{hairer2006}) and it is symplectic if $\Upsilon=1$.

The following important connection between symmetric ERKN methods
and symmetric trigonometric integrators has been given in
\cite{17-new}.
\begin{prop}\label{connection thm}
(See \cite{17-new}.)
 For the one-stage explicit symmetric ERKN methods and the symmetric trigonometric
integrator \eqref{TI}, the following connection is true
 \begin{equation}
\Phi_h=\Phi_{-h/2,\textmd{L}}\circ\Phi_{-h/2,\textmd{NL}}\circ\hat{\Phi}_h\circ
\Phi_{h/2,\textmd{NL}}\circ \Phi_{h/2,\textmd{L}}.
  \label{connetciton2}
\end{equation}
 Moreover, we have that
 \begin{equation}\begin{aligned}
\underbrace{\Phi_h\circ \cdots \circ\Phi_h}_{n\
\textmd{times}}=&\Phi_{-h/2,\textmd{L}}\circ\Phi_{-h/2,\textmd{NL}}\circ\big(\underbrace{\hat{\Phi}_h\circ
\cdots \circ \hat{\Phi}_h}_{n\ \textmd{times}}\big)   \circ
\Phi_{h/2,\textmd{NL}}\circ \Phi_{h/2,\textmd{L}}\\
=&\Phi_{h/2,\textmd{L}}\circ\Phi_{h/2,\textmd{NL}}\circ\big(\underbrace{\hat{\Phi}_h\circ
\cdots \circ \hat{\Phi}_h}_{n-1\ \textmd{times}}\big)   \circ
\Phi_{h/2,\textmd{NL}}\circ \Phi_{h/2,\textmd{L}}.
  \label{connetciton3}%
\end{aligned}\end{equation}
\end{prop}

 Since the long-time behaviour of  symplectic and symmetric trigonometric
integrators  is well studied for wave equations  in
\cite{Cohen08-1}, the same long term behaviour of one-stage explicit
symplectic and symmetric ERKN methods can be derived  by using the
relation \eqref{connetciton3} and by verifying the following two
points.
\begin{itemize}
\item I. When   \eqref{initi cond} holds, the   initial value  $$(\tilde{q}^0,\tilde{p}^0 ):=\Phi_{h/2,\textmd{NL}}\circ \Phi_{h/2,\textmd{L}}(q^0,p^0 )$$ for $\big(\underbrace{\hat{\Phi}_h\circ
\cdots \circ \hat{\Phi}_h}_{n-1\ \textmd{times}}\big)$ also
satisfies the   condition
\begin{equation}
\big(\norm{\tilde{q}^0}_{s+1}^2+\norm{\tilde{p}^0}_{s}^2\big)^{1/2}\leq C\epsilon. \label{new IV energy condition}%
\end{equation}

\item II. The two additional steps with $\Phi_{h/2,\textmd{L}}$ and $\Phi_{h/2,\textmd{NL}}$  only introduce an $\mathcal{O}(\epsilon) $ deviation in the
actions, momentum and  energy, provided that the corresponding
initial values $(\tilde{q},\tilde{p})$ are bounded by $\epsilon$.

\end{itemize}

 For the point I, in the light of  the splitting method, one has   that
\begin{equation*}
\begin{aligned} &\tilde{q}^0=\cos\big(\frac{h\Omega}{2}\big)q^0+\Omega^{-1}\sin\big(\frac{h\Omega}{2}\big)p^0,\\
&\tilde{p}^0=-\Omega \sin\big(\frac{h\Omega}{2}\big)
q^0+\cos\big(\frac{h\Omega}{2}\big)p^0+\frac{1}{2}h\Upsilon(\frac{h\Omega}{2})\tilde{g}(\tilde{q}^0)
.
\end{aligned}
\end{equation*}
Thus we get $\norm{\tilde{q}^0}_{s+1}^2\leq
\norm{q^0}_{s+1}^2+\norm{p^0}_{s}^2\leq  \epsilon^2$ and
\begin{equation*}
\begin{aligned}
&\norm{\tilde{q}^0}_{s+1}^2+\norm{\tilde{p}^0}_{s}^2\leq
2(\norm{q^0}_{s+1}^2+\norm{p^0}_{s}^2)+\frac{h^2}{4}
\norm{\tilde{g}(\tilde{q}^0)}_{s}^2\leq 2\epsilon^2+\frac{h^2}{4}
\norm{\tilde{g}(\tilde{q}^0)}_{s}^2.
\end{aligned}
\end{equation*}
Therefore \eqref{new IV energy condition} is clear  by considering
   the Lipschitz bound of $\tilde{g}$ (Sect. 4.2 in \cite{Hairer08} and
Sect. 6.10 in \cite{Cohen08-1}).

   For the second point,   according to
\begin{equation*}
\begin{aligned} \Phi_{h/2,\textmd{L}}\circ\Phi_{h/2,\textmd{NL}}
(\tilde{q},\tilde{p})=&    \left(  \begin{array}{cc}
                                                            \cos(\frac{h\Omega}{2})& \Omega^{-1}\sin(\frac{h\Omega}{2}) \\
                                                             -\Omega \sin(\frac{h\Omega}{2}) & \cos(\frac{h\Omega}{2}) \\
                                                           \end{array}
 \right)  \left(
             \begin{array}{c}
              \tilde{q} \\
              \tilde{p}+\frac{h}{2}\Upsilon(\frac{h \Omega}{2})\tilde{g}(\tilde{q}) \\
             \end{array}
           \right)\\
=&\left(
             \begin{array}{c}
             \cos(\frac{h\Omega}{2}) \tilde{q}+\Omega^{-1}\sin(\frac{h\Omega}{2})\tilde{p}+\frac{h^2}{4} \textmd{sinc}(\frac{h\Omega}{2})\Upsilon(\frac{h \Omega}{2})\tilde{g}(\tilde{q}) \\
              -\Omega \sin(\frac{h\Omega}{2}) \tilde{q}+ \cos(\frac{h\Omega}{2})\tilde{p}+\frac{h}{2} \cos(\frac{h\Omega}{2})\Upsilon(\frac{h \Omega}{2})\tilde{g}(\tilde{q}) \\
             \end{array}
           \right),
\end{aligned}
\end{equation*}
it can be checked that this method is of order two, which leads to
this point.

The proof of Corollary \ref{main theo3} is complete.

\section{Conclusions} \label{sec:conclusions}
In this work, we analysed the  long-time conservation properties of
 one-stage symplectic or symmetric
  ERKN methods when applied to nonlinear  wave equations via
spatial spectral semi-discretizations. We derived two main results
concerning the long-time near conservations of actions, momentum and
energy for symplectic or symmetric ERKN methods in the full
discretisation of nonlinear wave equations.  The technique of
multi-frequency
 modulated Fourier expansions
was used to prove the  main results by giving a modulated Fourier
expansion of symplectic or symmetric ERKN methods and showing three
almost-invariants of the modulation system. A relationship between
symplectic  and symmetric   ERKN methods and trigonometric
integrators was established to give another proof for the long-term
behaviour of symplectic and symmetric ERKN methods.

\section*{Acknowledgements}
The first author is grateful to Christian Lubich for the useful
comments on the topic of modulated Fourier expansions. The first
author is also very grateful to Ludwig  Gauckler for showing the
connection given in Section \ref{sec: ss proof}.

\end{document}